%% file: rha.tex
\documentclass[12pt]{amsart}

\usepackage{amssymb, amsthm, hyperref}
\usepackage{enumerate, amsfonts, amsxtra,amsmath,latexsym,epsfig, color}
\usepackage{cases}
\usepackage{thm-restate}

\usepackage{hyperref}

\usepackage{xypic}

\newcommand{\putinbox}[1]{\noindent\fbox{\parbox{\textwidth}{#1}}}

\usepackage{amsmath}
\newcommand{\rightQ}[2]{\left.\raisebox{.2em}{$#1$}\middle/\raisebox{-.2em}{$#2$}\right.}
\newcommand{\leftQ}[2]{\left.\raisebox{-.2em}{$#2$}\middle\backslash\raisebox{.2em}{$#1$}\right.}

\usepackage{enumitem}

\newcommand{\mc}{\mathcal}
\newcommand{\co}{\colon\thinspace}

\newtheorem{thmA}{Theorem}

\newtheorem {theorem}{Theorem}[section]
\newtheorem {lemma} [theorem] {Lemma}
\newtheorem {proposition} [theorem] {Proposition}

\newtheorem {corollary} [theorem] {Corollary}

\newtheorem {claim}{Claim}
\newtheorem* {claim*}{Claim}

\newtheorem* {subclaim*}{Subclaim}
\newtheorem {case}{Case}
\newtheorem {subcase}{Subcase}[case]

\newtheorem {notation} [theorem] {Notation}

\newtheorem {convention} [theorem] {Convention}

\newtheorem {assumption} [theorem] {Assumption}
\newtheorem {hypothesis} [theorem] {Hypothesis}

\theoremstyle{definition}
\newtheorem{remark}[theorem]{Remark}
\newtheorem {definition} [theorem] {Definition}
\newtheorem {example} [theorem] {Example}

\def\Z {\mathbb Z}

\def\R {\mathbb R}
\def\Hyp {\mathbb H}

\def\U {\mathfrak U}

\def\Stab{\mathrm{Stab}}

\def\CAT {\ensuremath{\operatorname{CAT}}}

\def\mc {\mathcal}

\def\onto {\, {\twoheadrightarrow}\, }

\newcommand{\Fone}{F_{\ref{lem:closetohyperplane}}}
\newcommand{\Ftwo}{F_{\ref{lem:attractivecarrier}}}

\def\Rgen {R}  
\def\Rfin {R_{\star}}

\newcommand{\defkappa}{\Fone(Q+1) + 2 Q + C + \frac{\sqrt{n}}{2} + 2}
\newcommand{\defD} {\max\left\{2\sqrt{n}\kappa+\sqrt{n},\ \Ftwo(0) + 6\eta + 2f(\delta) + 6\delta,\ 2C+1 \right\} + 1}
\newcommand{\deftau}{\Fone(Q)+\left(3+ \frac{1}{\sin(\frac{\theta_n}{2})}\right)Q + 4\eta+f(\delta) + 2\delta + \frac{\sqrt{n}}{2}}
\newcommand{\defomega}{2 \Ftwo(0) + \left(8+\frac{2}{\sin(\frac{\theta_n}{2})}\right)Q + 20\eta + 8f(\delta) + 16\delta + 2 \sqrt{n}}

\newcommand{\llangle}{\langle\negthinspace\langle}
\newcommand{\rrangle}{\rangle\negthinspace\rangle}

\newcommand{\acts}{\curvearrowright}

\newcommand{\hatx}{\widehat X}
\newcommand{\defeq}{:=}
\newcommand{\diam}{\operatorname{diam}}
\newcommand{\gprod}[3]{({#1}\,|\,{#2})_{#3}}
\begin{document}

\title{Specializing cubulated relatively hyperbolic groups}

\author[D. Groves]{Daniel Groves}
\address{Department of Mathematics, Statistics, and Computer Science,
University of Illinois at Chicago,
322 Science and Engineering Offices (M/C 249),
851 S. Morgan St.,
Chicago, IL 60607-7045}
\email{groves@math.uic.edu}

\author[J.F. Manning]{Jason Fox Manning}
\address{Department of Mathematics, 310 Malott Hall, Cornell University, Ithaca, NY 14853}
\email{jfmanning@math.cornell.edu}

\thanks{The first author was supported in part by NSF Grants DMS-1507067 and DMS-1904913.  The second author was in residence at ICERM for part of this work during the Illustrating Mathematics conference, supported by NSF grant DMS-1439786.  For a further part of this work he was visiting Cambridge University and thanks Trinity College and DPMMS for their hospitality.  Thanks also to the Simons Foundation for support under grant (\#524176, JFM). We thank the referees for helpful comments and corrections which improved this paper}

\begin{abstract}
  In \cite{VH}, Agol proved the Virtual Haken and Virtual Fibering Conjectures by confirming a conjecture of Wise:  Every cubulated hyperbolic group is virtually special.  We extend this result to cocompactly cubulated relatively hyperbolic groups with minimal assumptions on the parabolic subgroups.    Our proof proceeds by first recubulating to obtain an improper action with controlled stabilizers (a \emph{weakly relatively geometric} action), and then Dehn filling to obtain many cubulated hyperbolic quotients.  We apply our results to prove the Relative Cannon Conjecture for certain cubulated or partially cubulated relatively hyperbolic groups.

  One of our main results (Theorem~\ref{t:RH Agol}) recovers via different methods a theorem of Oreg\'on-Reyes \cite{OregonReyes}.  
\end{abstract}

\date{\today}

\maketitle
\section{Introduction}
If $G$ is a group with a proper, cellular cocompact action on a \CAT$(0)$ cube complex $X$, we say $G$ is \emph{cubulated by $X$}.
  (Some authors drop the cocompactness assumption.)  Actions on cube complexes typically arise via a construction of Sageev \cite{Sageev} from collections of \emph{codimension-one} subgroups.

A cube complex is \emph{special} if it admits a locally isometric immersion to the Salvetti complex of some right-angled Artin group (see \cite{HW08}, where this notion was introduced, and called \emph{A-special}). 
The group $G$ is \emph{specially cubulated by $X$} if the action can be chosen to be free, with quotient a special cube complex.  (We may also say that $G$ acts \emph{co-specially} on $X$.)  Such a group $G$ embeds into a finitely generated right-angled Artin group, so in particular it is linear over $\Z$ and hence residually finite.  It also inherits many useful separability properties from the right-angled Artin group.  

The group $G$ is \emph{virtually specially cubulated by $X$} if there is a finite index subgroup $G_0$ of $G$ which acts freely on $X$ with quotient a special, compact cube complex.  The $\Z$--linearity and nice separability properties of $G_0$ are passed on to $G$, so this is nearly as useful a property as being specially cubulated.  Indeed, the Virtual Haken and Virtual Fibering Theorems of Agol \cite{VH} are proved by showing that any hyperbolic group which is cubulated is virtually specially cubulated.

We note that not all cubulated groups are virtually specially cubulated.  Indeed Wise showed there are infinite cubulated groups with no finite index subgroups \cite{WiseCSC}; Burger and Mozes showed there are even infinite simple cubulated groups \cite{BurgerMozes97,BurgerMozes00}.  By taking free products of such groups, one can produce cubulated \emph{relatively} hyperbolic groups which are not residually finite and hence not virtually specially cubulated.  This suggests that to extend Agol's result about cubulated hyperbolic groups to relatively hyperbolic groups it is necessary to make some assumptions about the parabolic subgroups.

\subsection{Main results}
  The following appear to us to be the minimal possible assumptions on a relatively hyperbolic pair $(G,\mc{P})$ acting geometrically on a \CAT$(0)$ cube complex $X$, which might allow us to conclude that $G$ acts virtually co-specially on $X$.

\begin{assumption}\label{ass:weak}
  For each hyperplane stabilizer $S$ and each $P\in \mc{P}$ the intersection $S\cap P$ is separable in $P$. 
  \end{assumption}

\begin{assumption}  \label{ass:DCS}
 For each pair of hyperplane stabilizers  $S_1, S_2$ and each $P \in \mc{P}$, the double coset $\left( S_1 \cap P \right) \left( S_2 \cap P \right)$ is separable in $P$.
\end{assumption}

The following is our main result, which recovers via different methods a theorem of Oreg\'on-Reyes \cite{OregonReyes} (see Section~\ref{ss:otherwork} for the equivalence of our theorems).

\begin{restatable}{thmA}{RHAgol} \label{t:RH Agol}\cite[Theorem 1.2]{OregonReyes}
 Suppose that $(G,\mc{P})$ is relatively hyperbolic and that $G$ acts properly and cocompactly on a \CAT$(0)$ cube complex $X$ so Assumptions~\ref{ass:weak} and~\ref{ass:DCS} are satisfied.
 Then $G$ acts virtually co-specially on $X$. 
\end{restatable}

\begin{remark}\label{rem:minimal}
  Suppose that $(G,\mc{P})$ is relatively hyperbolic and that $G$ acts properly and cocompactly on a \CAT$(0)$ cube complex $X$.  It follows from Sageev--Wise \cite[Theorem 1.1]{SageevWise15} that for each $P \in \mc{P}$ there exists a convex $P$--cocompact sub-complex $Y_P$ in $X$.  It follows from Haglund--Wise \cite[Corollary 4.3]{HaglundWise10} that each $P$ acts virtually co-specially on $Y_P$ when $(G,\mc{P})$ satisfies Assumptions~\ref{ass:weak} and~\ref{ass:DCS}.  In fact the converse is also true using work from \cite[Appendix A]{OregonReyes}; see Subsection~\ref{ss:otherwork} for details.
  If $G$ acts virtually co-specially on $X$, then each $P$ also acts virtually co-specially on $Y_P$.  This justifies our belief that Assumptions~\ref{ass:weak} and~\ref{ass:DCS} are minimal.
\end{remark}

\begin{remark}
  In case $G$ is cubulated and the elements of $\mc{P}$ are virtually abelian, Assumptions~\ref{ass:weak} and~\ref{ass:DCS} always hold.  This applies in particular to cubulated fundamental groups of finite volume hyperbolic manifolds and orbifolds.
\end{remark}

We use the special case where elements of $\mc{P}$ are abelian to prove the Relative Cannon Conjecture for groups which are cubulated (Corollary~\ref{cor:cubeRC}) and for those which admit a weakly relatively geometric action on a \CAT$(0)$ cube complex (Theorem~\ref{t:wRG RC}).  This second result strengthens one of Einstein--Groves \cite{EinsteinG}.  See Definition~\ref{def:wrg} for the definition of a weakly relatively geometric action.

  While both assumptions hold whenever $G$ acts virtually co-specially, it is unclear whether Assumption~\ref{ass:DCS} is really necessary (see \cite[Problem 13.38]{Wise:ICM}).  In case the parabolics are hyperbolic, Assumption~\ref{ass:DCS} follows from Assumption~\ref{ass:weak}, by work of Minasyan \cite{minasyan:subsetgferf}.

  Much of our analysis does not depend on Assumption~\ref{ass:DCS}, and in particular under only Assumption~\ref{ass:weak} we are able to prove that various relatively quasi-convex subgroups are separable.  In particular Theorem~\ref{t:big fill separable} and Theorem~\ref{t:recubulate} have the following consequence.  We do not see how to obtain this result via Oreg\'on-Reyes' methods.
  \begin{thmA}\label{t:Hyp sep}
    Suppose $(G,\mc{P})$ is relatively hyperbolic and that $G$ acts properly cocompactly on a \CAT$(0)$ cube complex so that Assumption~\ref{ass:weak} is satisfied.  Then every hyperplane stabilizer is separable in $G$. 
  \end{thmA}
More generally Theorem~\ref{t:big fill separable} states that relatively quasi-convex subgroups are separable whenever their intersections with parabolics are separable in those parabolics.

\subsection{Relationship to work of Agol and Oreg\'on-Reyes}\label{ss:otherwork}
Formally, Theorem~\ref{t:RH Agol} generalizes Agol's main theorem in~\cite{VH}.  However, we use that theorem in an essential way to prove~\cite[Theorem D]{Omnibus}, which is used in an essential way in our proof of Theorem~\ref{t:RH Agol}.  In particular we do not give a new proof here of Agol's theorem.

We stated above that Oreg\'on-Reyes' main theorem in~\cite{OregonReyes} is equivalent to ours.  At first glance his hypotheses seem slightly weaker, but they are equivalent by a result proved in~\cite[Appendix A]{OregonReyes}, as we explain in this subsection.  We do not use this equivalence (Proposition~\ref{prop:equivalence}) anywhere in our paper, but provide it for the convenience of the reader.

Oreg\'on-Reyes uses the following terminology.
\begin{definition}
  Let $G$ be a group, acting properly and cocompactly on a cube complex $X$.  Then $(G,X)$ is a \emph{cubulated group}.  If $H<G$ acts properly cocompactly on some convex subcomplex $Y\subset X$ then $H$ is a \emph{convex subgroup}.
  
  The cubulated group $(G,X)$ is \emph{special} if $G$ acts freely and the quotient $\leftQ{X}{G}$ is special.  It is \emph{virtually special} if there is a finite index subgroup $G'<G$ so that $(G',X)$ is special.  
\end{definition}

The conclusions of~\cite[Theorem 1.2]{OregonReyes} and our Theorem~\ref{t:RH Agol} are the same.  However as we have mentioned the hypotheses look different.  Here is Oreg\'on-Reyes' hypothesis.
\begin{hypothesis}\label{hyp:or}
  $(G,X)$ is a cubulated group, and $(G,\mc{P})$ is relatively hyperbolic, where each $P\in \mc{P}$ is a convex subgroup of $G$, preserving a convex subcomplex $Y_P$ so that $(P,Y_P)$ is virtually special.
\end{hypothesis}
Our hypothesis can be stated as follows.
\begin{hypothesis}\label{hyp:gm}
  $(G,X)$ is a cubulated group, and $(G,\mc{P})$ is relatively hyperbolic, satisfying Assumptions~\ref{ass:weak} and~\ref{ass:DCS}.
\end{hypothesis}

\begin{proposition}\label{prop:equivalence}
  Hypotheses~\ref{hyp:or} and~\ref{hyp:gm} are equivalent.
\end{proposition}
\begin{proof}
  Assume that $(G,X)$ satisfies~\ref{hyp:or}.  If $S$ is a hyperplane stabilizer it is a convex subgroup of $G$.  
  By~\cite[A.15]{OregonReyes}, the subgroup $S\cap P$ is a convex subgroup of $(P,Y_P)$.  Thus by~\cite[A.1]{OregonReyes}, Assumptions~\ref{ass:weak} and~\ref{ass:DCS} hold.

  Now assume that $(G,X)$ satisfies~\ref{hyp:gm}.  As pointed out in Remark~\ref{rem:minimal}, each $P\in \mc{P}$ acts properly and cocompactly on some convex subcomplex $Y_P\subset X$ by a theorem of Sageev and Wise.  Under Assumptions~\ref{ass:weak} and~\ref{ass:DCS}, Haglund--Wise show that $(P,Y_P)$ is virtually special.
\end{proof}

\subsection{Our strategy}
  
Our proof of Theorem~\ref{t:RH Agol} has two steps.  The first, which takes up the bulk of this paper, is to replace the given cubulation of $G$ by a \emph{weakly relatively geometric} cubulation.  In the following definition, recall that a \emph{full} subgroup of a relatively hyperbolic group is one whose intersection with any parabolic subgroup is either finite or finite index in the parabolic.
\begin{definition} \label{def:wrg}
  If $(G,\mc{P})$ is relatively hyperbolic, and $X$ is a \CAT$(0)$ cube complex, an action $G\acts X$ is \emph{weakly relatively geometric} if the following hold:
  \begin{enumerate}
    \item each $P\in \mc{P}$ fixes some point of $X$;
    \item $G$ acts cocompactly on $X$; and
    \item\label{wrg:stabs}  if $\sigma$ is a cell of $X$ with infinite stabilizer, then
      \begin{enumerate}
      \item $\Stab(\sigma)$ is full relatively quasi-convex, and
      \item\label{wrg:stabs:gog} $\Stab(\sigma)$ is the fundamental group of a graph of groups, where each edge group is finite and each vertex group is either finite or full parabolic.
      \end{enumerate}
  \end{enumerate}
\end{definition}
The notion of a \emph{relatively geometric} action was defined in \cite{EinsteinG}; it can be obtained from the above by strengthening \eqref{wrg:stabs} to the requirement that every cell stabilizer is either finite or full parabolic.

Under Assumption~\ref{ass:weak} we put a new wallspace structure on the vertices of $X$.  We first find appropriate sub-complexes $\{X_P\mid P\in \mc{P}\}$ stabilized by the peripheral subgroups and having certain ``superconvexity'' properties.  Translating these around we get a $G$--equivariant family of \emph{parabolic} sub-complexes with uniformly bounded pairwise intersection.  For each ($G$--orbit of) hyperplane $H$ of $X$, we amalgamate the hyperplane stabilizer with finite index subgroups of the stabilizers of all the parabolic sub-complexes it meets, obtaining a relatively quasi-convex subgroup $B$.  Now we associate a wall $\mc{W}_H$ to $H$, declaring two vertices $v,w$ to be on the same ``side'' of this wall if there is a path joining them which crosses $B \cdot  H$ an even number of times.  These finitely many walls are completed to a $G$--equivariant wallspace structure $\mc{W}$ on $X^{(0)}$.  This is described in more detail in Section~\ref{sec:wallspace}.  In Section~\ref{sec:ultrafilters} we prove that the action on the cube complex dual to this wallspace structure is weakly relatively geometric, establishing:

\begin{restatable}{thmA}{recubulate}\label{t:recubulate}
  Suppose $(G,\mc{P})$ is relatively hyperbolic, and that $G$ acts properly cocompactly on a \CAT$(0)$ cube complex $X$, and that this action satisfies Assumption~\ref{ass:weak}.
  
There exists a weakly relatively geometric action of $G$ on a CAT$(0)$ cube complex $\hatx$.
\end{restatable}

The existence of a weakly relatively geometric action is the main input to the second step in our proof of Theorem~\ref{t:RH Agol}, which is completed in Section~\ref{sec:Dehn fill}.  The weakly relatively geometric action allows us to apply the tools in \cite{Omnibus} to find a rich family of hyperbolic virtually special quotients of $G$ (Theorem~\ref{t:quotient VS}).  We then use these quotients to prove separability of the hyperplane subgroups and double cosets in the original cubulation.  By a criterion of Haglund--Wise \cite[Corollary 4.3]{HaglundWise10}, this separability implies virtual specialness.
 
\subsection{Outline and conventions}
We assume the reader is familiar with the theory of relatively hyperbolic groups and relatively quasi-convex subgroups (see Hruska~\cite{HruskaQC}) and also the theory of special cube complexes (see \cite{RAAGs}).  We always use the locally \CAT$(0)$ metric on a non-positively curved cube complex, and notions such as distance and convexity refer to this metric.

In Section~\ref{sec:wallspace}, under the assumption that a relatively hyperbolic $G$ acts properly and cocompactly on a \CAT$(0)$ cube complex $X$, we recall and establish some basic properties of this action.  Under Assumption~\ref{ass:weak}, we define {\em augmented walls} from the hyperplanes of $X$ which will be used to define a new wallspace structure on $X$.  In Section~\ref{sec:ultrafilters} the cube complex associated to this wallspace structure is analyzed and Theorem~\ref{t:recubulate} is proved.  In Section~\ref{sec:Dehn fill} Dehn fillings are used to deduce Theorem~\ref{t:RH Agol}, as well as Theorem~\ref{t:Hyp sep}.   In Section~\ref{sec:examples} we provide an example which Theorem~\ref{t:RH Agol} implies is virtually special, but whose virtual specialness does not seem to follow from previously known results.  In Section~\ref{sec:Cannon} we provide the promised applications towards the Relative Cannon Conjecture.  Finally, in Appendix~\ref{app:DCS} we prove a technical generalization of a result from \cite{Hwide} which is required for the proof of Theorem~\ref{t:RH Agol}.

\begin{convention} \label{conv:P_infinite}
In this paper, whenever we speak of a relatively hyperbolic group pair $(G,\mc{P})$ we always assume that each element of $\mc{P}$ is infinite.  We may always ensure this by removing the finite elements from $\mc{P}$, which does not affect relative hyperbolicity.  The removal of these finite parabolics also does not affect whether a given action is weakly relatively geometric in the sense of Definition~\ref{def:wrg}.
\end{convention}

Note also that we use the notation $H^g = g H g^{-1}$.

\section{A new wallspace structure}\label{sec:wallspace}

\putinbox{Throughout this section and the next we fix a relatively hyperbolic group pair $(G,\mc{P})$ and a proper cocompact $G$--action on a \CAT$(0)$ cube complex $X$ satisfying Assumption~\ref{ass:weak}.}\medskip

The goal of this section and the next is to prove Theorem~\ref{t:recubulate}.  In the current section we describe a wallspace structure which gives rise to a new \CAT$(0)$ cube complex $\hatx$; in the next we prove that the $G$--action on $\hatx$ is weakly relatively geometric.  Our wallspace has underlying set $X^{(0)}$; the new walls come from amalgamations of the original hyperplanes with carefully chosen finite index subgroups of parabolic groups.  They are similar in spirit to the \emph{augmented hyperplanes} used in \cite{AGM_msqt}, but we do not work in an augmented cube complex as in that paper.

\begin{definition}\label{def:wall space}
A \emph{wall} in a set $S$ is a partition of $S$ into two nonempty subsets usually written $W=\{W^+,W^-\}$.  The two elements of the partition are called the \emph{halfspaces} from the wall.
A wall \emph{separates} $a$ from $b$ in $S$ if $a$ and $b$ lie in different halfspaces from the wall.
A \emph{space with walls} is a pair $(S,\mc{W})$ where $\mc{W}$ is a collection of nonempty subsets of $S$ which is closed under complementation and which satisfies the finiteness condition:
\begin{equation}\label{wallspacefiniteness}
  \forall x,y\in S,\ \#\left\{W\subseteq \mc{W}\mid W\mbox{ separates }x\mbox{ from }y\right\}<\infty  \tag{$\star$}
\end{equation}

Two walls $W = \{W^+,W^-\}$ and $V = \{V^+,V^-\}$ are said to \emph{cross} if for every choice of $\epsilon_V,\epsilon_W\in \{+,-\}$, the intersection $W^{\epsilon_W}\cap V^{\epsilon_V}$ is non-empty.  
\end{definition}
  This is a special case of a \emph{wallspace} as studied in \cite{hruskawise:finiteness} (see also the references therein, including \cite{Sageev,Nica,ChatterjiNiblo,HaglundPaulin}).  In the current paper we use the term ``wallspace'' in this paper interchangeably with ``space with walls''. 
A wallspace gives rise to a \CAT$(0)$ cube complex in a way which we revisit in the next section.  Conversely, the hyperplanes of the \CAT$(0)$ cube complex $X$ give a wallspace structure on $X^{(0)}$ in an obvious way.  In the current section we put a less obvious wallspace structure on $X^{(0)}$ which includes information about the peripheral structure.  To begin, we need to see that peripheral structure in the cube complex $X$.

\subsection{Peripheral complexes}
In this subsection we find some sub-complexes of $X$ associated to the peripheral subgroups $\mc{P}$ and record some of their properties.  In contrast to the peripheral sub-complexes used in \cite{AGM_msqt}, our peripheral complexes are not necessarily disjoint, though they have bounded overlap with each other.

The following definition is a slight variant of \cite[Definition 2.5]{Einstein-Hierarchies} (the difference being that we do not consider actually thin triangles to be relatively thin).

\begin{definition}
 Suppose that $M$ is a geodesic metric space, let $a,b,c \in M$ and let $\Delta(a,b,c)$ be a geodesic triangle.  Let $\pi \co \Delta(a,b,c) \to Y_{abc}$ be the map to the comparison tripod.  For $\nu \ge 0$, we say that $\Delta(a,b,c)$ is {\em $\nu$--thin} if for all $p \in Y_{abc}$ we have $\diam\left(\pi^{-1}(p)\right) \le \nu$.
 
 If $U \subseteq M$ is a subset then $\Delta(a,b,c)$ is {\em $\nu$--thin relative to $U$} if it is not $\nu$--thin and
 for every $p \in Y_{abc}$ either
\begin{enumerate}
 \item $\diam\left( \pi^{-1}(p) \right) \le \nu$; or
 \item $\pi^{-1}(p) \subseteq N_{\nu}(U)$.
\end{enumerate}
\end{definition}

\begin{definition}[Thin and fat parts of triangles] \label{def:triangle terminology}
  Let $\Delta(a,b,c)$ be a geodesic triangle in a metric space $M$, and let $\nu>0$.  Let $\pi\co \Delta(a,b,c) \to Y_{abc}$ be the map to the comparison tripod.  Let $o$ be the central point of the tripod.  The \emph{internal points} of $\Delta(a,b,c)$ are the points in $\pi^{-1}(o)$.
  
  The \emph{$\nu$--fat part} of $\Delta(a,b,c)$ is the union of those fibers of $\pi$ with diameter $\ge \nu$.
  
  Let $x\in\{a,b,c\}$, and let $l_x$ be the leg of the tripod joining $\pi(x)$ to $o$.  Let $t_x\subseteq l_x$ be a maximal connected subset containing $\pi(x)$ and so that $\diam\pi^{-1}(p)< \nu$ for all $p\in t_x$.  The set $\pi^{-1}(t_x)$ is a union of two segments starting at $x$, called the \emph{$\nu$--corner segments at $x$}.
  
\end{definition}
We remark that the $\nu$--fat part of a triangle is always closed.  If the metric on $M$ is convex (for example \CAT$(0)$) then any triangle is the union of its $\nu$--fat part and $\nu$--corner segments.
\begin{lemma}\label{lem:fatpart bounds thinness}
  If $\Delta$ is a triangle in a convex metric space which is $\nu$--thin relative to a set $Z$, and the fat part of some side has length at most $L$, then $\Delta$ is $2(L+\nu)$--thin.
\end{lemma}
\begin{proof}
  In the terminology of Definition~\ref{def:triangle terminology}, $\operatorname{insize}(\Delta)=\diam(\pi^{-1}(o))$.  In a convex metric space, $\operatorname{insize}(\Delta)\le K$ implies that $\Delta$ is $K$--thin, for any $K\ge 0$.  It therefore suffices to bound $\operatorname{insize}(\Delta)$.

  If the $\nu$--fat part of $\Delta$ is nonempty, then, using convexity of the metric, $\pi^{-1}(o)$ is in the fat part.  Let $s$ be the side of $\Delta$ whose fat part has length at most $L$, and let $x$ be the point of $\pi^{-1}(o)$ contained in $s$.  Let $y_1$ and $y_2$ be the other two points.  We have $d(x,y_1)\le 2 L_1+\nu$, and $d(x,y_2)\le 2L_2+\nu$, where $L_1+L_2 = L$.  Thus  $d(y_1,y_2)\le 2L+2\nu$, and $\operatorname{insize}(\Delta)\le 2L+2\nu$.
\end{proof}
\begin{definition}
 Let $M$ be a geodesic metric space and $\mc{U}$ a collection of subspaces.  We say that the pair $(M,\mc{U})$ has {\em $\nu$--relatively thin triangles} if for every geodesic triangle $\Delta(a,b,c)$, either $\Delta(a,b,c)$ is $\nu$--thin or else there exists $U \in \mc{U}$ so that $\Delta(a,b,c)$ is $\nu$--thin relative to $U$.
\end{definition}

The following is an immediate consequence of \cite[Proposition 2.7]{Einstein-Hierarchies} (see also the proofs of \cite[Theorem 4.1, Proposition 4.2]{SageevWise15} and \cite[$\S8$]{drutusapir}).

\begin{proposition}\label{prop:RH implies relatively thin}
Fix $x \in X$ and let $\mc{U} = \left\{ gP \cdot x \mid P \in \mc{P}, gP \in \rightQ{G}{P} \right\}$.  There exists $\nu$ so that $(X,\mc{U})$ has $\nu$--relatively thin triangles.
\end{proposition}

The following lemma is obvious.

\begin{lemma}\label{lem:Haus RT}
  Suppose $(M,\mc{U})$ has $\nu$--relatively thin triangles, and that $\mc{V}$ is a collection of subspaces so that every $U\in \mc{U}$ is contained in some $V\in \mc{V}$.  Then $(M,\mc{V})$ has $\nu$--relatively thin triangles.
\end{lemma}

The next theorem is an immediate consequence of Sageev--Wise \cite[Theorem 1.1]{SageevWise15}.

\begin{theorem} \label{th:SW core}
Let $x \in X$ be as in Proposition~\ref{prop:RH implies relatively thin}.  For each $P \in \mc{P}$ there exists a convex $P$--invariant sub-complex $Z_P \subseteq X$ with $x \in Z_P$ so that $\leftQ{Z_P}{P}$ is compact.
\end{theorem}
For each $P \in \mc{P}$ fix some $Z_P$ as in Theorem~\ref{th:SW core}, and let $\mc{Z} = \left\{ g \cdot  Z_P \mid P \in \mc{P}, g \right\}$, where $g$ ranges over a set of coset representatives of $\rightQ{G}{P}$.
Combining Proposition~\ref{prop:RH implies relatively thin} with Lemma~\ref{lem:Haus RT} and Theorem~\ref{th:SW core} yields the following.
\begin{corollary} \label{cor:RT}
There exists $\delta \ge 0$ so that $(X,\mc{Z})$ has $\delta$--relatively thin triangles.
\end{corollary}

\medskip
\putinbox{The constant $\delta$ from Corollary~\ref{cor:RT} will remain fixed for the remainder of this section and the next.}
\medskip

The following definition is a slight variant of \cite[Definition 5.1]{Einstein-Hierarchies}.
\begin{definition}
Let $\nu \ge 0$ and let $\phi \co \R_{\ge 0} \to \R_{\ge 0}$ be a function. Suppose that $M$ is a geodesic metric space, and that $\mc{U}$ is a collection of subspaces of $M$.  The pair $(M,\mc{U})$ is a {\em $(\nu,\phi)$--relatively hyperbolic pair} if
\begin{enumerate}
 \item $(M,\mc{U})$ has $\nu$--relatively thin triangles; and
 \item For all $r \ge 0$ and all $F_1, F_2 \in \mc{U}$ with $F_1 \ne F_2$, we have
 \[	\diam \left( N_r(F_1) \cap N_r(F_2) \right) \le \phi(r)	.	\]
\end{enumerate}
The subspaces $\mc{U}$ are called {\em peripheral subspaces}.
\end{definition}

The following is an immediate consequence of \cite[Theorems 4.1, 5.1 and A.1]{drutusapir} and Proposition~\ref{prop:RH implies relatively thin}.

\begin{lemma} \label{lem:RHP}
Let $x \in X$ and $\mc{U}$ be as in Proposition~\ref{prop:RH implies relatively thin}
and let $\delta$ be as in Corollary~\ref{cor:RT}.
There is a function $f_0 \co \R_{\ge 0} \to \R_{\ge 0}$ so that $(X,\mc{U})$ is a $(\delta,f_0)$--relatively hyperbolic pair.
\end{lemma}

The following easy fact is left to the reader.
\begin{lemma} \label{lem:Haus RHP}
 Let $M$ be a geodesic metric space and suppose that $\mc{U}$ is a set of subspaces so that $(M,\mc{U})$ is a $(\nu,\phi)$--relatively hyperbolic pair (for some $\nu \ge 0$ and function $\phi$).  Suppose further that $\mc{V}$ is a collection of subspaces of $M$ and that there is $r \ge 0$ and a bijection $\rho \co \mc{U} \to \mc{V}$ so that for all $U \in \mc{U}$ we have $U \subseteq \rho(U) \subseteq N_r(U)$.  Then $(M,\mc{V})$ is  $(\nu,\phi')$--relatively hyperbolic pair where $\phi'(x) = \phi(x+r)$.
\end{lemma}

The following is an immediate consequence of Lemmas~\ref{lem:RHP} and~\ref{lem:Haus RHP}.
\begin{corollary} \label{cor:RHP}
Let $\mc{Z}$ be as in Corollary~\ref{cor:RHP}, let $r_0$ be the maximum diameter of $\leftQ{Z_P}{P}$ for $P \in \mc{P}$, let $f_0$ be the function from Lemma~\ref{lem:RHP}, let $f_1$ be the function defined by $f_1(x) = f_0(x+r_0)$, and let $\delta$ be the constant from Corollary~\ref{cor:RT}. Then
$(X,\mc{Z})$ is a $(\delta,f_1)$--relatively hyperbolic pair.
\end{corollary}

The subspaces $\mc{Z}$ from Corollary~\ref{cor:RHP} are not the ones that we want.  Instead, we want them to satisfy the following condition,
which is a slightly stronger condition than that of `attractive' in \cite[Definition 5.2]{Einstein-Hierarchies} in that we insist that {\em most} of the geodesic $[a,b]$ is contained in $Z$, not just {\em some} of it.

\begin{definition}
Let $M$ be a complete {\CAT}(0) space and ${\eta} \ge 0$.  A convex subspace $Z \subseteq M$ is {\em ${\eta}$--super-attractive} if for any $K \ge 0$ and any pair of points $a,b \in N_K(Z)$, all but the initial and terminal segments of $[a,b]$ of length $({\eta}+K)$ are contained in $Z$.
\end{definition}
\begin{remark}\label{rem:attractiveness}
  Let $K\co \R_{\ge 0}\to \R_{\ge 0}$ be a function.  
  Einstein \cite[Definition 5.2]{Einstein-Hierarchies} defines a subspace $Z$ of a geodesic space to be \emph{$K$--attractive} if, for every $m \ge 0$, every geodesic with endpoints in the $m$--neighborhood of $Z$ either meets $Z$ or has length at most $K(m)$.  If $Z$ is $\eta$--super-attractive, it is not hard to see that $Z$ is $K$--attractive, for $K(m)=2(\eta+m)$.
\end{remark}

We now fix notation for closest point projections.
\begin{notation}
Suppose that $M$ is complete and \CAT$(0)$, and that $A$ is a closed convex subset of $M$.  We write  $\pi_A \co M \to A$ for the closest point projection map. 
\end{notation}

  \begin{lemma}\label{lem:projection quad}
    Suppose that $M$ is complete \CAT$(0)$, $\mc{C}$ is a collection of closed convex subspaces, and $(M,\mc{C})$ is $(\nu,\phi)$--relatively hyperbolic.  Let $Z\in \mc{C}$, let $\pi_Z\co M \to Z$ be closest point projection, and let $a,b\in X$.  Let $Q$ be the geodesic quadrilateral with vertices $\pi_Z(a),a,b,\pi_Z(b)$.
    Suppose
  \begin{equation*}
    \nu' = 4\nu + 2\phi(\nu),\ \nu'' = 4\nu+2\phi(\nu'),\mbox{and } \Delta(\nu,\phi)=\nu'+\nu''.
  \end{equation*}
    
    Then either $d(\pi_Z(a),\pi_Z(b))\le  \phi(\nu')+2\nu+3\nu'$ or the quadrilateral $Q$ is $\Delta(\nu,\phi)$--slim.\footnote{Recall a polygon is \emph{$\mu$--slim} if each side is contained in the $\mu$--neighborhood of the union of the other sides.}
  \end{lemma}
  \begin{proof}
    We assume that \begin{equation}\label{eq:projectionsfar} d(\pi_Z(a),\pi_Z(b)) > \phi(\nu')+\nu+3\nu'.\end{equation}

    We divide $Q$ by a diagonal into an ``upper'' triangle with vertices $a,b,\pi_Z(b)$ and a ``lower'' triangle with vertices $a,\pi_Z(a),\pi_Z(b)$.  See Figure~\ref{fig:slimquad}.
    \begin{figure}[htbp]
      \centering
      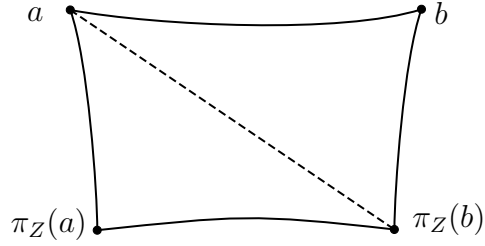
      \caption{Two points and their projections to $Z$ form the vertices of a slim quadrilateral.}
      \label{fig:slimquad}
    \end{figure}

  We show that $Q$ is $\Delta(\nu,\phi)$--slim by showing that the lower triangle is $\nu'$--thin and the upper triangle is $\nu''$--thin.

  We examine the lower triangle first.  We may suppose it is not $\nu$--thin.  There is therefore some $W\in \mc{C}$ so that it is $\nu$--thin relative to $W$.

  If $W=Z$, we argue as follows.  Let $a_1$ be the point of the fat part of $[a,\pi_Z(a)]$ closest to $a$.  Since $d(a_1,Z)\le \nu$, the segment $[a_1,\pi_Z(a)]$ can have length at most $\nu$.  The segment $[a_1,\pi_Z(a)]$ contains the $\nu$--fat part of $[a,\pi_Z(a)]$, and so the lower triangle is $4\nu$--thin by Lemma~\ref{lem:fatpart bounds thinness}.

  In case $W\ne Z$, we argue that the $\nu$--fat part of $[\pi_Z(a),\pi_Z(b)]$ is bounded.  Indeed, the entire segment is contained in $Z$, so the part which is also in a $\nu$--neighborhood of $W$ has diameter bounded by $\phi(\nu)$.  Since the $\nu$--fat part of one side of the lower triangle has length at most $\phi(\nu)$, Lemma~\ref{lem:fatpart bounds thinness} shows that the lower triangle is $2(\phi(\nu)+\nu)$--thin.  In either case, we conclude that the lower triangle is $4\nu+2\phi(\nu)$--thin.

  We now show the upper triangle is $\nu''$--thin.  Again, we may assume that the upper triangle is not $\nu$--thin, so it is $\nu$--thin relative to some $V\in \mc{C}$.  

  If $V=Z$, then we argue as for the lower triangle, concluding that the $\nu$--fat part of the side $[b,\pi_Z(b)]$ must have length at most $\nu$, so the upper triangle is $4\nu$--thin by Lemma~\ref{lem:fatpart bounds thinness}.

  If $V\ne Z$, we argue as follows.  We first note something about the lower triangle:  Since $[\pi_Z(a),\pi_Z(b)]$ lies entirely in $Z$, and $\pi_Z(a)$ is the closest point of $Z$ to $a$,
  the Gromov product $\gprod{a}{\pi_Z(b)}{\pi_Z(a)}$ is at most $\nu'$.  Let $x$ be the point on the diagonal $[a,\pi_Z(b)]$ which is in the preimage of the central point of the comparison tripod for the lower triangle.  The segment $[x,\pi_Z(b)]$ lies in $N_{\nu'}(Z)$ and has length in the interval $[d(\pi_Z(a),\pi_Z(b))-\nu',d(\pi_Z(a),\pi_Z(b))]$.

  Now consider the location of $x$ in the \emph{upper} triangle.  There are three cases, depending on whether $x$ is in the $\nu$--corner segment at $a$, the $\nu$--fat part, or the $\nu$--corner segment at $\pi_Z(b)$.

  Suppose $x$ lies in the $\nu$--corner segment at $a$.  Then the $\nu$--fat part of the upper triangle meets the diagonal in a subsegment of $[x,\pi_Z(b)]\subseteq N_{\nu'}(Z)$.  Since $V\ne Z$, this subsegment has length at most $\phi(\nu')$.  By Lemma~\ref{lem:fatpart bounds thinness}, the upper triangle is $\nu'' = (2\nu+2\phi(\nu'))$--thin.

  Suppose next that $x$ lies in the $\nu$--fat part of the upper triangle.  This $\nu$--fat part can only extend $\phi(\nu')$ past $x$.  Consider the point $w$ on the $\nu$--fat part of $[b,\pi_Z(b)]$ closest to $\pi_Z(b)$.  The corresponding point on $[a,\pi_Z(b)]$ is within $\nu'$ of $[\pi_Z(a),\pi_Z(b)]\subseteq Z$ (and exactly $\nu$ away from $w$), so $d(w,\pi_Z(b))=d(w,Z)\le \nu+\nu'$.  This is also the distance from the $\nu$--fat part of $[a,\pi_Z(b)]$ to $\pi_Z(b)$, so we have $d(x,\pi_Z(b))\le \phi(\nu')+\nu+\nu'$.  Since $d(x,\pi_Z(a))\le 2\nu'$, we have $d(\pi_Z(a),\pi_Z(b))\le \phi(\nu')+\nu+3\nu'$, a contradiction to \eqref{eq:projectionsfar}.

  Finally suppose $x$ lies in the corner segment of the upper triangle adjacent to $\pi_Z(b)$.  Let $x'$ be the point on $[\pi_Z(a),\pi_Z(b)]$ which is in the preimage of the central point of the comparison tripod for the lower triangle.  The lower triangle is $\nu'$--thin, so $d(x,x')\le\nu'$.  Thus the point $x'$ is at most $\nu+\nu'$ from the corner segment of $[b,\pi_Z(b)]$ adjacent to $\pi_Z(b)$.  Since $d(x',b)\ge d(\pi_Z(b),b)$, this corner segment must have length at most $\nu+\nu'$.  Since $d(x',a)\ge d(\pi_Z(a),a)$ and the lower triangle is $\nu'$--thin, the distance from $\pi_Z(a)$ to $x'$ is at most $\nu'$.  We conclude that $d(\pi_Z(a),\pi_Z(b))\le d(\pi_Z(a),x')+d(x',\pi_Z(b))\le \nu + 2\nu'$, again contradicting \eqref{eq:projectionsfar}.
  \end{proof}

\begin{corollary} \label{cor:bounded_proj}   
Suppose that $M$ is complete and \CAT$(0)$, $\mc{C}$ is a collection of convex subspaces, and $(M,\mc{C})$ is $(\nu,\phi)$--relatively hyperbolic.  There exists a constant $C(\nu,\phi)$, depending only on $\nu$ and $\phi$ so that for any distinct $Z,Z' \in \mc{C}$ we have
\[	\diam(\pi_Z(Z')) \le C(\nu,\phi)	.	\]
\end{corollary}
\begin{proof}
Let $\nu', \nu''$ and $\Delta(\nu,\phi)$ be as in Lemma~\ref{lem:projection quad}. We set $C(\nu,\phi) = \max\{ \phi(\nu') + 2\nu + 3\nu', 2\Delta(\nu,\phi) + \phi(\Delta(\nu,\phi)) \}$.

Fix $a,b \in Z'$.  We have to prove that $d(\pi_Z(a),\pi_Z(b)) \le C(\nu,\phi)$.  Let $Q$ be the geodesic quadrilateral with vertices $\pi_Z(a),a,b,\pi_Z(b)$.  By Lemma~\ref{lem:projection quad} either $d(\pi_Z(a),\pi_Z(b)) \le \phi(\nu') + 2\nu + 3\nu'$ or $Q$ is $\Delta(\nu,\phi)$--slim.  

In the first case, we are done.  Thus, suppose that $Q$ is $\Delta(\nu,\phi)$--slim.  Therefore, the geodesic $\pi_Z(a),\pi_Z(b)$ is contained in the $\Delta(\nu,\phi)$--neighborhood of the other three sides.  
Using \CAT$(0)$ geometry, we see that the only points on $[\pi_Z(a),\pi_Z(b)]$ which lie within $\Delta(\nu,\phi)$ of $[a,\pi_Z(a)]$ lie within $\Delta(\nu,\phi)$ of $\pi_Z(a)$.  Similarly, the only points on $[\pi_Z(a),\pi_Z(b)]$ lying within $\Delta(\nu,\phi)$ of $[b,\pi_Z(b)]$ lie within $\Delta(\nu,\phi)$ of $\pi_Z(b)$.  Since $[a,b] \subseteq Z'$, the diameter of the set of points on $[\pi_Z(a),\pi_Z(b)]$ lying within $\Delta(\nu,\phi)$ of $[a,b]$ is at most $\phi(\Delta(\nu,\phi))$.  Thus, in case $Q$ is $\Delta(\nu,\phi)$--slim, we have 
\[	d(\pi_Z(a),\pi_Z(b)) \le 2\Delta(\nu,\phi) + \phi(\Delta(\nu,\phi)),	\]
as required.
\end{proof}

The following is a strengthened version of the last assertion in \cite[Proposition 5.3, p.23]{Einstein-Hierarchies}.
\begin{proposition}\label{prop:peripheral}
There exists ${\eta}$ so that for each $P \in \mc{P}$ there exists a convex ${\eta}$--super-attractive $P$--invariant sub-complex $Y_P \subseteq X$ so that $Z_P \subseteq Y_P$ and $\leftQ{Y_P}{P}$ is compact. 
\end{proposition}
\begin{proof}
  Let $\mc{Z}$, $\delta$, and $f_1$ be as in the conclusions to Theorem~\ref{th:SW core} and Corollary~\ref{cor:RHP}.  That is, for each $P\in \mc{P}$, there is a convex $P$--invariant $P$--cocompact sub-complex $Z_P$ of $X$, and the family $\mc{Z}$ is the family of distinct translates of these $Z_P$ under the action of $G$.  Corollary~\ref{cor:RHP} says the pair $(X,\mc{Z})$ is a $(\delta,f_1)$--relatively hyperbolic pair.  
Let $\Delta_0 = \Delta(\delta,f_1)$ be the constant from Lemma~\ref{lem:projection quad}.
  
 If $E$ is a subset of a \CAT$(0)$ cube complex the \emph{combinatorial hull} $\operatorname{Hull}(E)$ is defined to be the intersection of the convex subcomplexes containing $E$.

The argument at the beginning of Section 5 of \cite{SageevWise15} can be adapted to show that there is an $S$ (depending only on $\dim(X)$, $\delta$, $f_1$ and $\Delta_0 = \Delta(\delta,f_1)$) so that
  \begin{equation}\label{eq:S} \operatorname{Hull}(N_{\Delta_0}(Z))\subseteq N_S(Z)
    \end{equation}
  for every $Z\in \mc{Z}$.  For $P\in \mc{P}$, we take $Y_P = \operatorname{Hull}(N_{\Delta_0}(Z_P))$.
  We set
  \begin{equation*}
    \eta = \frac{1}{2}\left(2S +f_1(\delta')+2\delta+3\delta'\right).
  \end{equation*}

  From now on we fix $P$, and set $Z=Z_P$, $Y=Y_P$.
  
  Let $K\ge 0$, and let $a,b\in N_K(Y)$.  By \eqref{eq:S}, $a,b\in N_{K+S}(Z)$.  Let $\pi_Z$ be the closest point projection to $Z$, and consider the geodesic quadrilateral $Q$ with vertices $\pi_Z(a),a,b,\pi_Z(b)$.

  Given Lemma~\ref{lem:projection quad}, we argue as follows.  In case
  \[d(\pi_Z(a),\pi_Z(b))\le f_1(\delta')+2\delta+3\delta',\]
  we have
  \begin{align*}
    d(a,b) & \le d(a,\pi_Z(a))+d(\pi_Z(a),\pi_Z(b))+d(\pi_Z(b),b)\\
           &  \le 2(K+S)+f_1(\delta')+2\delta+3\delta' \\
           & \le 2(K+\eta),
  \end{align*}
  and there is nothing to show.
  
  Otherwise, every point on $[a,b]$ is within $\Delta_0$ of some point on
  \[[a,\pi_Z(a)]\cup[\pi_Z(a),\pi_Z(b)]\cup[b,\pi_Z(b)].\]  Let $[a',b']\subseteq [a,b]$ be the subsegment beginning and ending on \[N_{\Delta_0}([a,\pi_Z(a)]\cup[b,\pi_Z(b)]).\]    The lengths of the remaining subsegments $[a,a']$  and $[b,b']$ are at most $K+S+\Delta_0$.
  Since $[\pi_Z(a),\pi_Z(b)]\subseteq Z$, the subsegment $[a',b']$ lies in $N_{\Delta_0}(Z)\subseteq Y$, as desired.
\end{proof}

\begin{definition} \label{def:fix (X,B)}
  For each $P\in \mc{P}$, fix a $P$--invariant sub-complex $Y_P$ of $X$ satisfying the conclusions of Proposition~\ref{prop:peripheral}.  Any sub-complex of $X$ of the form $gY_P$ for $g\in G, P\in \mc{P}$ is called a \emph{peripheral complex}.  We let $\mc{B}$ be the collection of all peripheral complexes in $X$.
\end{definition}

The following is an immediate consequence of Corollary~\ref{cor:RHP} and Lemma~\ref{lem:Haus RHP}.
\begin{lemma}[cf. {\cite[Proposition 5.3]{Einstein-Hierarchies}}]\label{lem:(X,B) RHP} 
Let $\delta$ be as in Corollary~\ref{cor:RT} and define the function $f \co \R_{\ge 0} \to \R_{\ge 0}$ by $f(x) = f_1(x+r_1)$, where $f_1$ is the function from Corollary~\ref{cor:RHP} and $r_1$ is the maximum diameter of the spaces $\leftQ{Y_P}{P}$ from Proposition~\ref{prop:peripheral}.  Then the pair $(X,\mc{B})$ is $(\delta,f)$--relatively hyperbolic.  
\end{lemma}

\medskip
\putinbox{For the remainder of this section and the next we fix the collection $\mc{B}$ of convex $\eta$--super-attractive peripheral complexes from Definition~\ref{def:fix (X,B)} so that $(X,\mc{B})$ is a $(\delta,f)$--relatively hyperbolic pair, for $\eta$ as in Proposition~\ref{prop:peripheral} and $\delta, f$ as in Lemma~\ref{lem:(X,B) RHP}.}
\medskip

\subsection{Augmented walls}
Up to the $G$--action, there are finitely many hyperplanes in $X$.  In this subsection we fix some such hyperplane $H$, and describe a partition $\mc{W}_H=\{W_H^+,W_H^-\}$ of $X^{(0)}$ associated to that hyperplane together with some extra data we now specify.

\begin{definition}
 Suppose that $H$ is a hyperplane in $X$.
Let $\Stab(H)$ be the stabilizer of $H$, and let $Y_1, \ldots , Y_k$ be representatives of the $\Stab(H)$--orbits of peripheral complexes that intersect $H$.  

A {\em broadcaster} for $H$ is a subgroup of $G$ of the form 
\[	B = \langle \Stab(H), B_1, \ldots, B_k \rangle,	\]
where for each $i$ the subgroup $B_i$ is a finite-index subgroup of $\Stab(Y_i)$.  Given a broadcaster $B$, the associated collection $B \cdot H$ of hyperplanes of $X$ is called the {\em scattering of $H$ by $B$}.  Given $\Rgen > 0$, we say that $B \cdot H$ is an {\em $\Rgen$--separated scattering} if any two carriers of distinct hyperplanes in $B \cdot H$ are distance at least $\Rgen$ from each other.
\end{definition}

We make use of the following result of Mart\'inez-Pedroza, which we have slightly rephrased.
\begin{theorem}\label{thm:eduardo}\cite[Theorem 1.1]{MP09}
  Let $A<G$ be relatively quasi-convex, and let $P \in \mc{P}$.  There is a finite set $F \subseteq P\setminus A$ so that if $\dot{P}<P$ satisfies $A\cap P\subseteq \dot{P}\subseteq P\setminus F$, then:
  \begin{enumerate}
    \item $\widehat{A}:=\langle A,\dot{P}\rangle$ is relatively quasi-convex;
    \item $\widehat{A}= A\ast_{A\cap P}\dot{P}$; and
    \item\label{paraboliccontrol} Every parabolic subgroup of $\widehat{A}$ is $\widehat{A}$--conjugate either to a subgroup of $\dot{P}$ or a subgroup of $A$.
  \end{enumerate}
\end{theorem}

We will apply Theorem~\ref{thm:eduardo} to hyperplane stabilizers.  This requires the (well-known) fact that they are relatively quasi-convex.
\begin{lemma}\label{lem:stabH_RQC}
  Let $H$ be a hyperplane of $X$.  Then $\Stab(H)$ is relatively quasi-convex in $(G,\mc{P})$.
\end{lemma}
\begin{proof}
  Since the hyperplane $H$ is convex in $X$, there is a quasi-isometry of pairs $(X,H) \to (G,\Stab(H))$.  This implies $\Stab(H)$ is undistorted in $G$.  Therefore by \cite[Theorem 1.5]{HruskaQC} $\Stab(H)$ is relatively quasi-convex in $(G,\mc{P})$.
\end{proof}

\begin{proposition}\label{prop:amalgam}
  Let $H$ be a hyperplane in $X$, let $\{Y_1,\ldots,Y_k\}$ be as above, and let $\Rgen>0$.

There exists a broadcaster $B$ for $H$ so that $B$ is full relatively quasi-convex in $G$ and the associated scattering $B \cdot H$ is $\Rgen$--separated.
  \end{proposition}
\begin{proof}
  We use Theorem~\ref{thm:eduardo} to construct a sequence of relatively quasi-convex subgroups $\Stab(H)=A_0<A_1<\cdots<A_k = B$ so that $A_i$ is an amalgam of $A_{i-1}$ with a finite index subgroup of $\Stab(Y_i)$.
  
  By Lemma~\ref{lem:stabH_RQC}, $A_0$ is relatively quasi-convex in $(G,\mc{P})$.
  We inductively assume that $A_{i-1}$ is relatively quasi-convex and that
  \begin{equation}\label{eq:intersections}
    A_{i-1}\cap \Stab(Y_j) = A_0\cap \Stab(Y_j),\ \forall j\ge i.\tag{$\dagger_{i-1}$}
  \end{equation}  
  Applying Theorem~\ref{thm:eduardo} with $A=A_{i-1}$ and $P=\Stab(Y_i)$, we obtain $F_i = F \subseteq \Stab(Y_i)\setminus A_{i-1}$ as in the theorem.
  Assumption~\ref{ass:weak} guarantees the existence of a finite index $B_i<\Stab(Y_i)$ which contains $A_0\cap \Stab(Y_i)$ but misses $F_i$.  We can thus apply Theorem~\ref{thm:eduardo} with $\dot{P} = B_i$ to obtain a relatively quasi-convex $A_i=A_{i-1}\ast_{A_0\cap \Stab(Y_i)}B_i$.

  In order to continue the induction we must establish Condition $(\dagger_i)$; this follows easily from \cite[Lemma 5.4]{MP09}. It also follows from that result that the final amalgam $B = A_k$ is full.
  
  It remains to show that by expanding the finite sets $F_i$ if necessary we can ensure that the scattering $B \cdot H$ is $\Rgen$--separated.  Since $B \cdot H$ consists of the $B$--orbit of a single hyperplane, it suffices to ensure that the distance from $H$ to $B\cdot H \setminus H$ is at least $\Rgen+1$.  For fixed $\Rgen$, there exists a finite set $\{ g_1, \ldots , g_l \} \subseteq G \smallsetminus \Stab(H)$ so that any hyperplane within $\Rgen+1$ of $H$ is of the form $a \cdot g_i \cdot H$ for some $a \in \Stab(H)$.  Since $\Stab(H) \le B$, if $a \cdot g_i \cdot H \in B \cdot H$ then $g_i \cdot H \in B \cdot H$ also.  Therefore, we need to exclude the finite set $\{g_1 \cdot H, \ldots , g_l \cdot H\}$ from $B \cdot H$.  It is easy to see (again using $\Stab(H) \le B$) that this is the same as excluding the finite set $\{ g_1, \ldots , g_l \}$ from $B$.  That this can be done (by choosing $F_i$ appropriately, depending on $\Rgen$) is an easy consequence of \cite[Lemma 3.6 and Lemma 5.2]{MP09} -- we find a short geodesic between $1$ and $g_i$ which has no large $\dot{P}$--components, but on the other hand if $g_i \in B$ then there is some path in $B$ which is not contained in $\Stab(H)$ and necessarily contains a large $\dot{P}$--component.  This contradicts \cite[Lemma 3.6]{MP09}.  This argument completes the proof of the proposition.
\end{proof}

We are now ready to describe the augmented walls ${W}_H$.  There is some freedom in this description; further constraints are imposed in Assumption~\ref{ass:R} and Definition~\ref{def:wallspaceW} where we pick the particular augmented walls for our wallspace.
\begin{definition}\label{def:singlewall}
  Given a hyperplane $H$, and $\Rgen > 0$, fix a broadcaster $B$ whose associated scattering $B \cdot H$ is $\Rgen$--separated as in Proposition~\ref{prop:amalgam}.  Define an equivalence relation on $X^{(0)}$ by declaring $v\sim w$ if there is a path in the $1$--skeleton of $X$ which crosses $B \cdot  H$ an even number of times.  Let $W_{H,B}$ be the pair of equivalence classes.
\end{definition}
We observe:
\begin{lemma}\label{lem:stabhalfspace}
  Let $B$ and $W_{H,B}=\{W_{H,B}^+,W_{H,B}^-\}$ be as above.  Both $W_{H,B}^+$ and $W_{H,B}^-$ are nonempty, $\Stab(W_{H,B})=B$, and $\Stab(W_{H,B}^+)=\Stab(W_{H,B}^-)$ is a subgroup of $B$ of index at most two.

  In particular, $\Stab(W_{H,B})$ is full relatively quasi-convex in $(G,\mc{P})$.
\end{lemma}

\begin{remark}
  The hyperplane $H$ descends to an embedded hyperplane of $\leftQ{X}{B}$.  The index of $\Stab(W_{H,B}^+)$ in $B$ is one if this hyperplane is separating; otherwise the index is two.
\end{remark}

\subsection{Connected quasi-convex sub-complexes associated to walls}
Later we fix an $\Rgen$--separated scattering $B \cdot H$ of $H$ with $\Rgen$ very large.  Therefore, though $B \cdot H$ is quasi-convex, the quasi-convexity constant is very large.  
In order to associate a connected \emph{uniformly} quasi-convex subspace of $X$ to the wall $W_{H,B}$, we adjoin peripheral complexes to the hyperplane carriers.

There are three results in this subsection used extensively in the sequel.  In Proposition~\ref{prop:uniformqccarrier} we establish the uniform quasi-convexity of the \emph{thick carriers} defined below in Definition~\ref{def:thickcarrier}.  In Lemma~\ref{lem:attractivecarrier} we quantify the statement that peripheral complexes which come near to a thick carrier for a long time must be contained in the thick carrier.  In Lemma~\ref{lem:closetohyperplane}, we quantify the fact that two peripheral complexes in a thick carrier can only come near to each other if they are both near to a hyperplane of the scattering.  It is crucial for our applications that all these results are uniform for $\Rgen$--scatterings, so long as $\Rgen$ is sufficiently large.

\begin{definition}[Thick carrier]\label{def:thickcarrier}
Let $H$ be a hyperplane in $X$, $\Rgen > 0$ and $B \cdot H$ an $\Rgen$--separated scattering of $H$.

Let $\mc{A}$ be the collection of carriers of translates $bH$ where $b\in B$.  As before $\mc{B}$ is the set of peripheral complexes.  Form an abstract cube complex 
\begin{equation*}
  Y_{H,B}= \sqcup\mc{A}\bigsqcup\sqcup\mc{B}/\sim.
\end{equation*}
The equivalence relation $\sim$ is generated by the following identifications:  For each $bH$ with $b\in B$ and carrier $A\in \mc{A}$, and each $Z\in \mc{B}$ meeting $bH$,  glue $A$ to $Z$ along their intersection in $X$.  Now let $S_{H,B}$ be the component of $Y_{H,B}$ containing $H$.  (The other components are just elements of $\mc{B}$.)  There is a canonical map $\phi_{H,B}\co S_{H,B}\to X$ which restricts to the inclusion on any hyperplane carrier or peripheral complex.  We define the {\em thick carrier of $W_{H,B}$} to be $T(W_{H,B}) = \phi_{H,B}(S_{H,B})$.  The thick carrier $T(W_{H,B})$ is a union of convex sub-complexes.
\end{definition}

We observe:
\begin{lemma}
  $T(W_{H,B})$ is a connected sub-complex of $X$ on which $B$ acts properly cocompactly.
\end{lemma}

The next result uses work of Einstein \cite{Einstein-Hierarchies} to show that $T(W_{H,B})$ is quasi-isometrically embedded and quasi-convex.

\begin{proposition}\label{prop:qiecarrier}
  There exist constants $R_0 \ge 0$, $\lambda \ge 1$, and $\epsilon \ge 0$ depending only on $(X,\mc{B})$ so that the following holds.  Let
    $\Rgen > R_0$,
let $H$ be a hyperplane, and let $B \cdot H$ be an $\Rgen$--separated scattering of $H$.  Let $S = S_{H,B}$ and $\phi_{H,B}$ be as in Definition~\ref{def:thickcarrier}.  Then for any $p,q \in S$ we have
   \begin{equation}\label{qie}\tag{$*$}
     d_X(\phi_{H,B}(p),\phi_{H,B}(q))\le d_S(p,q)\le \lambda\, d_X(\phi_{H,B}(p),\phi_{H,B}(q))+\epsilon.
   \end{equation}
\end{proposition}
\begin{proof}
  We would like to apply Einstein's \cite[Proposition 5.16]{Einstein-Hierarchies}, with $\mc{A}$ equal to the carriers of hyperplanes in the $\Rgen$--scattering $B \cdot H$, and $\mc{B}_0$ equal to the collection of peripheral complexes meeting some $A\in \mc{A}$.  We have fixed a $(\delta,f)$ so that our pair $(X,\mc{B})$ is $(\delta,f)$--relatively hyperbolic, but we have to temporarily modify $f$ in order to apply Einstein's results.
  This is allowable because $(X,\mc{B})$ is also $(\delta,f')$--relatively hyperbolic for any function $f'$ which dominates $f$ (in the sense that $f'(r)\ge f(r)$ for all $r$).
   
Hypothesis 5.5 of \cite{Einstein-Hierarchies} requires that the pair $(X,\mc{B})$ is $(\delta,f')$--relatively hyperbolic, and every $B\in \mc{B}$ is $K$--attractive, where $K(m)=3 f'(5\delta) + 21\delta + 6m$.  Each $B\in \mc{B}$ is $\eta$--super-attractive.  Thus by Remark~\ref{rem:attractiveness}, each $B\in \mc{B}$ is $K_0$--attractive, where $K_0(m)=2\eta + 2m$.  As long as $K$ dominates $K_0$, each $B$ will also be $K$--attractive.  In order to ensure $K$ dominates $K_0$, we set
  \begin{equation*}
    f'(r) = \max\{\eta,f(r)\}.
  \end{equation*}
  Since $f'$ dominates $f$, the pair $(X,\mc{B})$ is still $(\delta,f')$--relatively hyperbolic.

  Now \cite[Hypotheses 5.15]{Einstein-Hierarchies} amounts to the assertion that $\mc{A}$ is an $\Rgen$--scattering for $\Rgen \ge 500f'(5\delta)+10000\delta$.  We therefore take
  \begin{equation*}
    R_0 = 500f'(5\delta)+10000\delta.
  \end{equation*}
  Then \cite[Proposition 5.16]{Einstein-Hierarchies} implies that $\phi_{H,B}$ is a $(\lambda,\epsilon)$--quasi-isometry, for $(\lambda,\epsilon) = (2,110f'(5\delta) + 1592\delta)$.  These constants only depend on $(X,\mc{B})$.  Since $\phi_{H,B}$ is obviously distance non-increasing, we have the inequalities \eqref{qie}.
\end{proof}

\smallskip
\putinbox{We fix the constants $R_0,\lambda,\epsilon$ from Proposition~\ref{prop:qiecarrier} for the remainder of this section and the next.}
\medskip

The next lemma says that peripheral complexes in a thick carrier can only be close to each other if they are close to a hyperplane associated to the carrier.

\begin{lemma}\label{lem:closetohyperplane}
  There exists a function $\Fone\co \R_{\ge 0} \to \R_{\ge 0}$, depending only on $(X,\mc{B})$, so that the following holds whenever $R>R_0$.

  Let $H$ be a hyperplane, let $B$ be a broadcaster for $H$ so that the associated scattering $B \cdot H$ is $\Rgen$--separated, and let $T(W_{H,B})$ be the corresponding thick carrier.
  For any $m \ge 0$ and any distinct peripheral complexes $Z_1,Z_2$ in $T(W_{H,B})$, there is some $g\in B$ so that
  \begin{equation*}
    N_m(Z_1)\cap N_m(Z_2) \subseteq N_{\Fone(m)}(gH).
  \end{equation*}
\end{lemma}
\begin{proof}
  Fix $p\in N_m(Z_1)\cap N_m(Z_2)$.
  For $i=1,2$, let $p_i$ be a point of $Z_i$ with $d_X(p_i,p)\le m$, and let $\hat{p}_i$ be the corresponding point of the complex $S_{H,B}$.  Using the second inequality of \eqref{qie} in Proposition~\ref{prop:qiecarrier}, we have $d_S(\hat{p}_1,\hat{p}_2)\le \lambda d_X(p_1,p_2) + \epsilon \le 2\lambda m + \epsilon$.
  Choose $gH$ so that the hyperplane carrier for $gH$ separates $Z_1$ from $Z_2$ in the complex $S_{H,B}$.
  Any path from $\hat{p}_1$ to $\hat{p}_2$ in $S_{H,B}$ must pass through the carrier of $gH$, so in particular $d_S(\hat{p}_1,gH)\le 2\lambda m+\epsilon+\frac{1}{2}$.  Thus (using the first inequality of \eqref{qie}) the distance between $p$ and $gH$ is at most $2\lambda m+\epsilon+m+\frac{1}{2}$.  Since the diameter of $N_m(Z_1)\cap N_m(Z_2)$ is at most $f(m)$, we can set $\Fone(m) = 2\lambda m+\epsilon+m+\frac{1}{2}+ f(m)$.
\end{proof}

For the next two results we prove, we need another result from Einstein \cite{Einstein-Hierarchies}.  Note that the result there uses peripheral cosets, whereas we will use the peripheral complexes $\mc{Z}$.  Since their Hausdorff distance is finite, Proposition~\ref{prop:fellow_travel} follows from \cite[Theorem 4.2]{Einstein-Hierarchies}.

\begin{definition}[Relative fellow traveling]
  Let $l\ge 0$.
  Let $\sigma_1$ and $\sigma_2$ be two paths in $X$ with the same endpoints, and suppose that these paths can be subdivided into an equal number of sub-paths, occurring in the same order.  Suppose further that if $J_{\sigma_1}\subseteq \sigma_1$ and $J_{\sigma_2}\subseteq\sigma_2$ are corresponding subpaths (called \emph{paired subpaths}), then they begin and end within $l$ of one another and that either
  \begin{enumerate}
    \item  the Hausdorff distance between $J_{\sigma_1}$ and $J_{\sigma_2}$ is at most $l$; or
    \item  there exists some $Z\in\mc{B}$ so $N_l(Z)$ contains both $J_{\sigma_1}$ and $J_{\sigma_2}$.
  \end{enumerate}
  Then $\sigma_1$ and $\sigma_2$ are said to \emph{$l$--relatively fellow travel, relative to $\mc{B}$.}  
\end{definition}

Quasi-geodesics in spaces with relatively thin triangles relatively fellow travel one another:
\begin{proposition} (cf. \cite[Theorem 4.2]{Einstein-Hierarchies}, \cite[Proposition 4.1.6]{hruskakleiner})\label{prop:fellow_travel}
  Let $\mu\ge 1$ and $\varkappa \ge 0$.   There exists $l$, depending only on $\mu$, $\varkappa$, and the pair $(X,\mc{B})$, so that $(\mu,\varkappa)$--quasi-geodesics with the same endpoints in $X$ must  $l$--relatively fellow travel, relative to $\mc{B}$.
\end{proposition}

The following lemma says that if a geodesic spends any time in a peripheral complex, then a quasi-geodesic with the same endpoints must spend a similar amount of time close to that peripheral complex.  (Recall that $\lambda,\epsilon$ are fixed constants coming from Proposition~\ref{prop:qiecarrier}.  If they were allowed to vary the lemma would still be true, but $l'$ would have to be allowed to depend on them.)
\begin{lemma}  \label{lem:l_prime}
  There is an $l'$ depending only on $(X,\mc{B})$ so that the following holds.  Let $\gamma$ be a geodesic in $X$, and let $\sigma$ be a $(\lambda,\epsilon)$--quasi-geodesic with the same endpoints as $\gamma$.  Let $Z$ be some peripheral complex so that $\gamma'=\gamma\cap Z$ is nonempty.  Then there is a subpath $\sigma'$ of $\sigma$ whose endpoints are within $l'$ of the endpoints of $\gamma'$, and which is contained in an $l'$--neighborhood of $Z$.
\end{lemma}

  \begin{proof}
   Let $l$ be the constant of relative fellow travelling from Proposition~\ref{prop:fellow_travel} applied to $(\lambda,\epsilon)$--quasi-geodesics.

   We claim first that there are points on $\sigma$ within
   \begin{equation*}
     l_1= \max\{2l+\eta,f(l)+l\}
   \end{equation*}
   of the endpoints of $\gamma'$.  Indeed, let $a$ be an endpoint of $\gamma'$, and let $J_\gamma$, $J_\sigma$ be paired subpaths of $\gamma$ and $\sigma$ so that $J_\gamma$ contains $a$.
  There are three possibilities:
   \begin{enumerate}
     \item\label{fellowtravel} $J_\gamma$ and $J_\sigma$ are Hausdorff distance at most $l$ from each other;
     \item\label{WisZ} $J_\gamma$ and $J_\sigma$ are both contained in the $l$--neighborhood of $Z$, or
     \item\label{WnotZ} $J_\gamma$ and $J_\sigma$ are both contained in the $l$--neighborhood of some peripheral complex $W\ne Z$.
   \end{enumerate}
   In case \eqref{fellowtravel} we find a point of $\sigma$ at most $l$ away from $a$.
   In case \eqref{WisZ}, we note that any component of $J_\gamma\setminus Z$ has diameter at most $l+\eta$, by $\eta$--super-attractiveness.  Thus an endpoint of $J_\gamma$ is at most $l+\eta$ from $a$.  The corresponding endpoint of $J_\sigma$ is at most $2l + \eta$ from $a$.
   In case \eqref{WnotZ}, we note that $J_\gamma\cap Z$ has diameter at most $f(l)$, so there is an endpoint of $J_\gamma$ at most $f(l)$ from $a$.  The corresponding endpoint of $J_\sigma$ is at most $f(l) + l$ from $a$.
   In any case we have found a point of $\sigma$ at most $l_1$ away from $a$.

   We let $\sigma'$ be a subsegment of $\sigma$ whose endpoints are at most $l_1$ from the endpoints of $\gamma'$.  Let $\beta$ be the path obtained from $\sigma'$ by adjoining geodesics between the endpoints of $\gamma'$ and those of $\sigma'$.  Then $\beta$ is a $(\lambda, \epsilon + 4l_1)$--quasi-geodesic with the same endpoints as $\gamma$.  Let $l_2$ be the constant of relative fellow travelling from Proposition~\ref{prop:fellow_travel} applied to $(\lambda, \epsilon + 4l_1)$--quasi-geodesics.
  We claim that $\beta$ (and hence $\sigma'$) lies in an $l_3$--neighborhood of $Z$, where 
  \begin{equation*}
    l_3 = \frac{\lambda^2}{2}\left(f(l_2)+2l_2+\epsilon + 4l_1\right) + \epsilon + 4l_1 + l_2.
  \end{equation*}
  Indeed, let $b$ be a point of $\beta$, and let $J_{\gamma'}$ and $J_\beta$ be paired subpaths of $\gamma'$ and $\beta$ so that $b\in J_\beta$.  As before there are three possibilities:
  \begin{enumerate}
    \item\label{betafellowtravel} $J_{\gamma'}$ and $J_\beta$ are Hausdorff distance at most $l_2$ from each other;
    \item\label{betaWisZ} $J_{\gamma'}$ and $J_\beta$ are both contained in the $l_2$--neighborhood of $Z$, or
   \item\label{betaWnotZ} $J_{\gamma'}$ and $J_\beta$ are both contained in the $l_2$--neighborhood of some peripheral complex $W\ne Z$.
  \end{enumerate}
  In cases \eqref{betafellowtravel} and \eqref{betaWisZ} there is clearly a point of $Z$ within $l_2$ of $b$.
  In case \eqref{betaWnotZ}, 
  the diameter of $N_{l_2}(W)\cap N_{l_2}(Z)$ is at most $f(l_2)$.  Thus the segment $J_{\gamma'}$ has length at most $f(l_2)$.  The endpoints of $J_\beta$ are thus at most $f(l_2)+2l_2$ apart from one another.  A computation shows that $b$ is at most
  \begin{equation*}
    \frac{\lambda^2}{2}\left(f(l_2)+2l_2+\epsilon + 4l_1\right) + \epsilon + 4l_1
  \end{equation*}
  from one of these endpoints, and thus at most $l_3$ from some point on $\gamma'\subseteq Z$.

  Setting $l' = l_3$, we have established the lemma.
  \end{proof}

The next lemma shows that a thick carrier which has large coarse intersection with a peripheral complex must contain that peripheral complex.
\begin{lemma}\label{lem:attractivecarrier}
There exists a function $\Ftwo\co \R_{\ge 0} \to \R_{\ge 0}$ and a constant $R_1\ge R_0$ depending only on $(X,\mc{B})$ so that the following holds.
Let $B \cdot H$ be an $\Rgen$--separated scattering of some hyperplane $H$, where $\Rgen>R_1$.  Let $T = T(W_{H,B})$ be the associated thick carrier, and let $Z$ be a peripheral complex.

For any $d \ge 0$, if
\[ \diam(Z\cap N_d(T)) \ge \Ftwo(d), \]
        then $Z \subseteq T$.
\end{lemma}
\begin{proof}
We take $R_1 = \max\{R_0, f(l')\}$, where $l'$ is the constant from Lemma~\ref{lem:l_prime}, and suppose $\Rgen\ge R_1$.
  Our proof is based on the following claim.
  \begin{claim*}
    There is an $L_0$ depending only on $(X,\mc{B})$ so that if $a,b\in T$ and $[a,b]$ has a subsegment of length $> L_0$ in some peripheral complex $Z$, then $Z\subseteq T$.
  \end{claim*}
  Given the claim, we argue as follows.  
  Fix $d>0$.  We define $\Ftwo(d) = L_0+4d+2\eta$.
  Suppose that $x$ and $y$ satisfy the following three properties:
  \begin{enumerate}
  \item $x,y\in Z$;
  \item $d(x,y) > \Ftwo(d)$; and
  \item $d(x,T)$ and $d(y,T)$ are at most $d$.
  \end{enumerate}
  Let $x'$, $y'$ be closest points on $T$ to $x$, $y$ respectively.  We have $d(x',y') > L_0+2d+2\eta$.  By $\eta$--super-attractiveness, a subsegment of $[x',y']$ of length greater than $L_0$ lies in $Z$.  Applying the Claim with $a=x',b=y'$, we conclude $Z\subseteq T$.

  It remains to prove the claim.
  \begin{proof}[Proof of Claim]
    We prove the claim with $L_0 = 2 f(l') + 4 l' + 2\eta$.
    Let $T$ be the thick carrier as in the statement of the lemma and $S = S(W_{H,B})$ the associated abstract complex.  Let $\sigma$ be the $(\lambda,\epsilon)$--quasi-geodesic joining $a$ to $b$ which is the image of a geodesic in $S$.  Suppose that $\gamma = [a,b]$ is the $X$--geodesic with the same endpoints and that $\gamma$ intersects the peripheral complex $Z$ in a segment $\gamma'$ of length greater than $L_0$.
    Lemma~\ref{lem:l_prime} gives a subpath $\sigma'$ of $\sigma$ starting and ending within $l'$ of the endpoints of $L_0$, so that $\sigma'$ is contained in the $l'$--neighborhood of $Z$.
    
  This subpath $\sigma'$ must have length bigger than $L_0 - 2l' = 2 f(l') + 2l' + 2\eta$; it is a broken geodesic made alternately of segments inside peripheral complexes and carriers of hyperplanes.

    If there are four or more such subsegments, one is a peripheral segment between two different hyperplane carriers in the $\Rgen$--scattering; in particular it has length at least $\Rgen> f(l')$.  But since the endpoints of this peripheral segment are within $l'$ of $Z$, the peripheral segment must actually be a $Z$--segment, and so $Z$ was in $T$ to begin with.  A similar argument shows $Z\subseteq T$ if $\sigma'$ begins or ends with a peripheral segment of length bigger than $f(l')$.

    If there is a hyperplane segment of length greater than $2l' + 2\eta$, then the $\eta$--super-attractiveness of $Z$ implies that this hyperplane segment meets $Z$.  In particular a hyperplane of the scattering meets $Z$, and again we conclude that $Z\subseteq T$.

    The only remaining possibility is that $\sigma'$ consists of three or fewer segments, at most one of which is a (short) hyperplane segment.  The total length of $\sigma'$ is thus at most $2f(l') + 2l' + 2\eta$, contrary to our choice of $L_0$.  This completes the proof of the claim.
  \end{proof}
Once the claim is proved, the lemma is proved.\end{proof}

\begin{proposition}\label{prop:uniformqccarrier}
  There is a $Q$ depending only on $(X,\mc{B})$ so that the following holds.
  Let $\Rgen > R_1$, where $R_1$ is the constant from Lemma~\ref{lem:attractivecarrier}.  Let $B \cdot H$ be an $\Rgen$--separated scattering of some hyperplane $H$ of $X$.  The thick carrier $T(W_{H,B})$ is $Q$--quasi-convex in $X$.
\end{proposition}
\begin{proof}
  We prove the Proposition with $Q = 2l + \eta + \frac{1}{2}\Ftwo(2l+\eta)$, where $l$ is the constant of relative fellow-traveling of $(\lambda,\epsilon)$--quasi-geodesics in $(X,\mc{B})$, $\eta$ is the super-attractiveness constant of elements of $\mc{B}$, and $\Ftwo$ is the function from Lemma~\ref{lem:attractivecarrier}.

  Let $T = T(W_{H,B})$, and let $S = S(W_{H,B})$.  Recall that $T$ is the image of $S$ under a $(\lambda,\epsilon)$--quasi-isometric embedding (Proposition~\ref{prop:qiecarrier}).
  Let $\gamma$ be a geodesic with endpoints in $T$, and let $\sigma$ be the image of an $S$--geodesic with the same endpoints.  Since $S$ is $(\lambda,\epsilon)$--quasi-isometrically embedded, $\sigma$ is a $(\lambda,\epsilon)$--quasi-geodesic in $X$.  In particular it $l$--relatively fellow travels $\gamma$ relative to $\mc{B}$.
  
  Let $J_\gamma$ and $J_\sigma$ be paired subpaths of $\gamma$ and $\sigma$ respectively.  In case they are Hausdorff distance $l$ from each other, we clearly have $J_\gamma\subseteq N_Q(T)$.  Otherwise, they are both in the $l$--neighborhood of some peripheral complex $Z$.   If the length of $J_\gamma$ is at most $2(l+\eta)+\Ftwo(2l+\eta)$, it is also contained in the $Q$--neighborhood of $T$, since its endpoints are distance at most $l$ from the endpoints of $J_\sigma\subseteq T$.  Otherwise, the $\eta$--super-attractiveness of $Z$ implies that there is a subsegment of $J_\gamma$ of length at least $\Ftwo(2l+\eta)$, whose endpoints are distance at most $l + \eta$ from the endpoints of $J_\gamma$, hence at most $2l + \eta$ from the endpoints of $J_\sigma$.  Lemma~\ref{lem:attractivecarrier} then shows that $Z\subseteq T$, so $J_\gamma\subseteq N_Q(T)$ in this case as well.
\end{proof}

\subsection{The wallspace}
We now construct the wallspace which we use to build the cubulation for Theorem~\ref{t:recubulate}.  In order to do this, we first fix some constants, and our assumptions on the constant of separatedness of the scatterings we use.

\begin{notation}\label{not:understandable}
Let $\eta$ be as in Proposition~\ref{prop:peripheral} and suppose that $\mc{B}$ is a collection of convex $\eta$--super-attractive complexes as in Definition~\ref{def:fix (X,B)}.  Let $\delta$ and $f$ be so that $(X,\mc{B})$ is $(\delta,f)$--relatively hyperbolic, as in the conclusion of Lemma~\ref{lem:(X,B) RHP}.  Let $n = \dim(X)$.  

We now recall and fix some more constants, and note that like $\eta,\delta,f$ and $n$, all of them depend only on the pair $(X,\mc{B})$.

The constant $R_1$ previously appeared in Lemma~\ref{lem:attractivecarrier} and Proposition~\ref{prop:uniformqccarrier}.
We will assume that all scatterings are $\Rgen$--scatterings for some $\Rgen>R_1$.

\begin{enumerate}
\item $C = C(\delta,f)$ is the bound on the diameter of projections of peripheral complexes to one another (see Corollary~\ref{cor:bounded_proj}).
\item $Q$ is the constant of quasi-convexity of thick carriers of scatterings (see Proposition~\ref{prop:uniformqccarrier});
\item  $\theta_n = \arcsin(\frac{1}{\sqrt{n}})$ (a geodesic in $X$ which is not too short must meet some hyperplane at this angle or greater, see Lemma~\ref{lem:sageevwise} below).
\end{enumerate}
\end{notation}
\begin{notation}\label{not:mysterious}
  The following constants also depend only on the pair $(X,\mc{B})$.  They may seem somewhat mysterious but we need to choose them carefully now in order to fix $\Rfin$ below.
  \begin{enumerate}
  \item (first used in Lemma~\ref{lem:weaklyprincipal});
    \begin{equation*}
      \kappa = \defkappa
    \end{equation*}
  \item  (first used in Lemma~\ref{lem:NJU})
    \begin{equation*}
    \omega = \defomega;
  \end{equation*}
\item (first used in Lemma~\ref{lem:NJU})
  \begin{equation*}
    \tau = \deftau;
    \end{equation*}
\item (first used in Definition~\ref{def:reln})    \begin{equation*}
    D = \defD
    \end{equation*}
\end{enumerate}
\end{notation}
\begin{assumption}\label{ass:R}
  Fix once and for all some number $\Rfin$ so that  $\Rfin$ is strictly bigger than all of the following:
\begin{enumerate}
\item\label{req:qiec and uniformqcc} (Required so Propositions~\ref{prop:qiecarrier} and~\ref{prop:uniformqccarrier} may be applied)
\begin{equation*}
R_1
\end{equation*}
\item\label{req:for unboundedNU} (required in Proposition~\ref{prop:unboundedNU})
\begin{equation*}
    \Ftwo(0)+4(\eta+\delta)+2f(\delta)+\tau
\end{equation*}
\item\label{req:for prop:eq rel} (required in Proposition~\ref{prop:eq rel})
\begin{equation*}
8(Q+\Fone(Q)+2\eta+\delta+\frac{1}{2}f(\delta))
\end{equation*}
\item\label{req:for prop:eq rel.2} (required in Proposition~\ref{prop:eq rel})
\begin{equation*}
\frac{8}{3} \left( D + 2 \eta + 2 \delta + f(\delta) \right)
\end{equation*}
\item\label{req:for equivclassboundedset} (required in Lemma \ref{lem:equivclassboundedset})
\begin{equation*}
D + \tau
\end{equation*}
\item\label{req:for Lambda tree} (required in Lemma~\ref{lem:Lambda_tree})
\begin{equation*}
\frac{8}{3} \left( 26 \cdot \max\left\{ f(5\delta),\eta,\frac{1}{3}\max \left\{ \omega, 2(D+\tau) \right\} \right\} + 250\delta \right)
\end{equation*}
\end{enumerate}  
\end{assumption}
\medskip\noindent\putinbox{The constant $\Rfin$ is the same for the remainder of this section and the next; we only consider $\Rfin$--scatterings from now on.}\medskip

\begin{definition}[The wallspace $(X^{(0)},\mc{W})$]\label{def:wallspaceW}  Choose representatives $H_i$ of the finitely many $G$--orbits of hyperplanes of $X$, and for each $i$  choose a broadcaster $B_i$ so that $B_i \cdot H_i$ is an $\Rfin$--separated scattering.  Such a broadcaster exists by Proposition~\ref{prop:amalgam}.  Let $W_i = W_{H_i,B_i}$ be the wall from Definition~\ref{def:singlewall} and let $\phi_i = \phi_{H_i,B_i}$ be the map defined in Definition~\ref{def:thickcarrier}.  It follows that $T(W_i)$, $\phi_i$ satisfy the conclusion of Propositions~\ref{prop:qiecarrier} and~\ref{prop:uniformqccarrier}.  Let \[\mc{W}=\{gA\mid g\in G, A\in W_i\mbox{ for some }i\}.\]
\end{definition}
Since $\mc{W}$ is defined $G$--equivariantly, there is an action of $G$ on the wallspace $(X^{(0)},\mc{W})$.

The following is immediate from the fact that walls are unions of hyperplanes in $X$ and from the equivalence relation defined in Definition~\ref{def:singlewall}.
\begin{lemma}\label{lem:edgedetermines}
  Let $v,w$ be adjacent vertices of $X$.  Then there is exactly one wall of $\mc{W}$ separating $v$ from $w$.
\end{lemma}

\begin{definition}\label{def:WHyp and WPC}
  Lemma~\ref{lem:edgedetermines} says in particular that each hyperplane of $X$ is associated to a single wall $W$ of $\mc{W}$.  Such a hyperplane is called a \emph{$W$--hyperplane}.  Let $W = gW_{H_i,B_i}$ for some $g\in G$ and one of the hyperplanes $H_i$ discussed in Definition~\ref{def:wallspaceW}.   We write $T(W)$ for the complex $gT(W_{H_i,B_i})$ and note that it is $Q$--quasi-convex and a $(\lambda,\epsilon)$--quasi-isometric image of $S_{H_i}$ (see Proposition~\ref{prop:qiecarrier}).  We refer to $T(W)$ as the \emph{thick carrier for $W$}.
  If $Z$ is one of the peripheral complexes used to build $S_{H_i}$, then $gZ$ is called a \emph{$W$--peripheral complex}.
\end{definition}
\begin{lemma} \label{lem:W separates}
  Let $Z$ be a peripheral complex and let $W$ be a wall.  Then $Z$ is a $W$--peripheral complex if and only if $W$ separates two vertices of $Z$.
\end{lemma}

\begin{lemma}
  The collection $\mc{W}$ defined in Definition~\ref{def:wallspaceW} makes $X^{(0)}$ into a space with walls.
\end{lemma}
\begin{proof}
  We must show the finiteness condition \eqref{wallspacefiniteness} from Definition~\ref{def:wall space}.  But Lemma~\ref{lem:edgedetermines} implies that the number of walls in $\mc{W}$ separating two vertices is at most their distance in $X^{(1)}$.
\end{proof}

\section{Ultrafilters and Stabilizers for the action}\label{sec:ultrafilters}
The goal of this section is to prove Theorem~\ref{t:recubulate}.
To that end, we fix a wallspace structure $(X^{(0)},\mc{W})$ as described in Definition~\ref{def:wallspaceW}.  By Sageev's construction this gives rise to a new \CAT$(0)$ cube complex $\hatx=C(\mc{W})$ with a $G$--action.  Many properties of this cube complex (for instance dimension and the precise vertex stabilizers) depend strongly on the choices made in the last section.  However, as long our cube complex is built from $\Rfin$--scatterings where $\Rfin$ satisfies Assumption~\ref{ass:R}, this cube complex will satisfy the conclusions of Theorem~\ref{t:recubulate}.

We briefly recall the Sageev construction (see \cite{Sageev,SageevPCMI,hruskawise:finiteness} for more details).  Given a wallspace $(S,\mc{V})$, an \emph{ultrafilter} on $(S,\mc{V})$ is a subset $\U$ of $\mc{V}$ satisfying the following two properties:
\begin{enumerate}
  \item (Consistency) If $A\subseteq B$ are two halfspaces and $A\in \U$, then $B\in \U$; and
  \item (Completeness) If $A$ is a halfspace, then either $A$ or $A^c$ is in $\U$.
\end{enumerate}
A \emph{DCC ultrafilter} is one which contains no infinite descending chain
\[ A_1\supsetneq A_2\supsetneq\cdots.\]
\begin{convention}
  All ultrafilters considered below are assumed to be DCC.
\end{convention}
An ultrafilter is \emph{principal} if, for some $x\in S$, it is of the form:
\[ \U[x] \defeq \{A\in \mc{V}\mid x\in A\}.\]

\begin{notation}
  Completeness means we can also think of an ultrafilter $\U$ as a function from the set of walls to the set of halfspaces.  Abusing notation slightly, if $W=\{A,A^c\}$ is a wall and $\U$ is an ultrafilter so $A\in \U$, we write $\U(W) = A$.
\end{notation}

We build a cube complex $C'(\mc{V})$ as follows.  The $0$--cells are the DCC ultrafilters on $\mc{V}$.  Two ultrafilters are connected by an edge if they differ on a single wall.  (In other words $\U_1 = (\U_2\setminus\{A\})\cup\{A^c\}$.)  Higher dimensional cubes are added wherever their $1$--skeleta appear.  The principal ultrafilters all lie in the same component of $C'(\mc{V})$, and it is this component which we name $C(\mc{V})$.  We say $C(\mc{V})$ is \emph{the cube complex obtained by applying the Sageev construction to $\mc{V}$}.

The following result is essentially due to Sageev, though he did not phrase things in terms of wallspaces.  See \cite[Section 3]{hruskawise:finiteness}, and the references therein.  See also \cite{SageevPCMI,Roller} for a slightly more general discussion of the cube complex dual to a ``pocset'' (poset with complementation).
\begin{theorem}\cite{Sageev}
  For any wallspace $(S,\mc{V})$, the cube complex $C(\mc{V})$ is \CAT$(0)$.
\end{theorem}

\subsection{Parabolic sub-complexes of $\hatx$ are finite}
In \cite[Section 3.4]{hruskawise:finiteness}, Hruska and Wise introduce the notion of a \emph{hemiwallspace}, a subcollection of the halfspaces of a wallspace containing at least one from every wall.  When $(S,\mc{V})$ is a wallspace and $T\subseteq S$, one obtains the \emph{hemiwallspace induced by $T$} (see \cite[Example 3.20]{hruskawise:finiteness}):
\begin{equation*}
  \mc{V}_T = \{U\in \mc{V}\mid U\cap T\ne \emptyset\}.
\end{equation*}
This hemiwallspace also gives rise to a cube complex $C(\mc{V}_T)$ defined to be the cube complex obtained by applying the Sageev construction to the collection of halfspaces in $\mc{V}_T$ whose complements are also in $\mc{V}_T$.
We have the following result of Hruska--Wise.
\begin{lemma}\cite[Lemmas 3.23, 3.24]{hruskawise:finiteness}\label{lem:hemi}
  The cube complex $C(\mc{V}_T)$ embeds naturally into $C(\mc{V})$ as a convex sub-complex.
\end{lemma}

We apply this result with $S=X^{(0)}$, $\mc{V} = \mc{W}$, and $T = Y^{(0)}$ for $Y$ equal to some peripheral complex.  In this case we have:  
\begin{proposition}\label{prop:finiteperipheralsubcomplex}
  $C(\mc{W}_{Y^{(0)}})$ is a finite cube complex.
\end{proposition}
\begin{proof}
  We show that there are only finitely many walls $W\subseteq \mc{W}$ so that $W\subseteq \mc{W}_{Y^{(0)}}$.  The Sageev construction applied to a wallspace with finitely many walls yields a finite cube complex, so this establishes the proposition.

  Suppose that $W=\{W^+,W^-\}\subseteq \mc{W}_{Y^{(0)}}$.  This means both $W^+\cap {Y^{(0)}}$ and $W^-\cap {Y^{(0)}}$ are nonempty.  We first show that $W$ is determined by the partition $$\{W^+\cap {Y^{(0)}},W^-\cap {Y^{(0)}}\}.$$  Indeed, since $Y$ is connected, there is a pair of adjacent vertices $v,w\in {Y^{(0)}}$ lying in different halfspaces.  But Lemma~\ref{lem:edgedetermines} shows that $W$ is the only wall in $\mc{W}$ separating $v$ from $w$.  

  We next claim that there is a fixed finite index subgroup of $\Stab(Y)$ which preserves the partition $\{W^+\cap {Y^{(0)}},W^-\cap {Y^{(0)}}\}$ for every $W\subseteq \mc{W}_{Y^{(0)}}$.  Each such partition is is preserved by some finite index subgroup of $\Stab(Y)$ conjugate in $G$ to one of the subgroups of the $\Stab(Y_j)$ in the definition of the broadcasters for the $H_i$.
 There are only finitely many such subgroups of $\Stab(Y)$, so some finite index subgroup $N_Y$ of $\Stab(Y)$ is contained in all of them.  In particular, $W$ descends to a partition of $\leftQ{{Y^{(0)}}}{N_Y}$, and is determined by that partition.

  Since there are only finitely many partitions of a finite set, the set of walls in $\mc{W}_{Y^{(0)}}$ is finite.
\end{proof}
Work of Hruska--Wise \cite{hruskawise:finiteness} gives the following.
\begin{corollary}\label{cor:cocompact}
  The action of $G$ on $C(\mc{W})$ is cocompact.  In particular $C(\mc{W})$ is finite dimensional.  Moreover each $P\in \mc{P}$ fixes a point of $C(\mc{W})$.
\end{corollary}
\begin{proof}
  By Lemma~\ref{lem:stabhalfspace}, the wall stabilizers for $G\acts \mc{W}$ are full relatively quasi-convex.
  
  Theorem 7.12 of \cite{hruskawise:finiteness} (with $\heartsuit = \frac{1}{4}$) shows that for $(G,\mc{P})$ relatively hyperbolic acting on a wallspace with relatively quasi-convex wall stabilizers, the $G$--action on the associated cube complex is ``relatively cocompact.''  Since the sub-complexes associated to the peripheral subgroups are compact (cocompact under the action of the peripheral subgroups is enough), the $G$--action is  cocompact.

  Let $P\in \mc{P}$.  Then $P$ preserves a peripheral complex $Y\subseteq X$, so it preserves the convex sub-complex $C(\mc{W}_{Y^{(0)}})\subseteq \hatx$.  As this is a complete bounded diameter \CAT$(0)$ space (Proposition~\ref{prop:finiteperipheralsubcomplex}), there is a fixed point for the action of $P$ (see \cite[II.2.8]{bridhaef:book}).
\end{proof}

\subsection{Associating a characteristic subspace of  $X$ to an ultrafilter}
In general not every ultrafilter $\U$ on $\mc{W}$ is principal.  In other words the intersection of all the halfspaces in $\U$ may be empty.  We show in this subsection that by appropriately augmenting and thickening the halfspaces of $\U$ we do obtain a nonempty intersection (see Proposition~\ref{prop:NofU}).

\begin{definition}[Halfspace carriers]
  In Definition~\ref{def:WHyp and WPC} we defined thick carriers for walls.
  If $W=\{A,A^c\}$, then define the \emph{halfspace carrier} $\overline{A}$ to be the full sub-complex of $X$ on $A\cup \left(T(W)\cap X^{(0)}\right)$.
\end{definition}
The full sub-complex on a halfspace $A$ is not connected, but the halfspace carrier is connected, as the following lemma shows.  We will use the observations in this lemma repeatedly in the rest of the section.
\begin{lemma}
  Let $A\in \mc{W}$ be a halfspace.  Then $\overline{A}$ is a connected sub-complex of $X$.  Moreover,  $\overline A\cap \overline{A^c}=T(W)$ and $\overline A\cup \overline{A^c} = X$.
\end{lemma}
\begin{proof}
  Every component of the full subcomplex on $A$ meets $T(W)$, which is connected.  Thus $\overline{A}$ is connected.

  The intersection of full subcomplexes is a full subcomplex, so $\overline{A}\cap \overline{A^c}$ is the full subcomplex on the zero-skeleton of $T(W)$.  This is $T(W)$.
  
  Any cube of $X$ with all its vertices in $A$ is in the full subcomplex on $A$, so it is contained in $\overline{A}$.  Similarly any cube of $X$ with all its vertices in $A^c$ is contained in $\overline{A^c}$.  A cube with vertices in both $A$ and $A^c$ must be in the carrier of some hyperplane of the scattering associated to $W$, hence in $T(W)$.  So $\overline{A}\cup\overline{A^c} = X$.
\end{proof}

Next we show that halfspace carriers have the same ``attractiveness'' property as we showed for thick carriers in Lemma~\ref{lem:attractivecarrier}.
\begin{lemma}\label{lem:attractive halfspace carriers}
  Let $m > 0$ and let $\Ftwo(m)$ be as in Lemma~\ref{lem:attractivecarrier}.  Let $\overline{A}$ be a halfspace carrier and let $Z$ be a peripheral complex.  If
\[ \diam(Z\cap N_m(\overline{A}))\ge \Ftwo(m), \]
then $Z\subseteq \overline{A}$.
\end{lemma}

\begin{proof}
  Let $W=\{A,A^c\}$ be the associated wall and let $T = T(W)$.  If $Z$ contains vertices of both $A$ and $A^c$, it must meet some $W$--hyperplane, so $Z\subseteq T\subseteq \overline{A}$.  If $Z$ contains only vertices of $A$ then $Z\subseteq \overline{A}$.

  We are left with the case that every vertex of $Z$ is in $A^c$.
  Let $x,y\in Z\cap N_m(\overline{A})$.
  Since $x,y$ are not in $A$, a shortest path to $\overline{A}$ must pass through $T$, so they are in $N_m(T)$.  Using the case of wall carriers, if $d(x,y)\ge \Ftwo(m)$, then $Z\subseteq T\subseteq \overline{A}$, and we are done.
\end{proof}

\begin{lemma}\label{lem:nestingcarriers}
  Suppose $W,V$ are  walls of $\mc{W}$ so that $V^+\subseteq W^+$.  Then $T(V)\subseteq \overline{W^+}$.
\end{lemma}
\begin{proof}
  It is enough to show that if $v\in T(V)^{(0)}$, then $v\in \overline{W^+}$.  If $v$ is in the carrier of a $V$--hyperplane, it is distance $\le 1$ from a vertex in $V^+$.  Since the $W$--hyperplanes and the $V$--hyperplanes are disjoint, this means $v$ must actually be in $W^+$.

  Suppose then that $v$ is contained in a $V$--peripheral complex $Z$, and that $v\in W^-$.  Since $Z$ is a $V$--peripheral complex, by Lemma~\ref{lem:W separates} it must contain vertices in both $V^+$ and $V^-$.  In particular it contains some $w\in V^+\subseteq W^+$.  Since $Z$ contains vertices of both $W^+$ and $W^-$, it must also be a $W$--peripheral complex, and so $Z\subseteq \overline{W^+}$.
\end{proof}

We next prove a lemma about coarse intersections of thick carriers of crossing walls.
\begin{lemma}\label{lem:wallsintersect}
  For each positive integer $m$ there exists a $\Theta_m\ge 0$ so that for any collection of pairwise crossing walls $\{W_0,\ldots,W_m\}$ of $\mc{W}$,
  \[ \bigcap_{i=0}^m N_{\Theta_m}(T(W_i))\ne \emptyset.\]
\end{lemma}
\begin{proof}
  The proof is by induction, so we start by establishing the base case $m=1$.  We prove the contrapositive with $\Theta_1 = 0$.
  Suppose that there are walls $W_0,W_1$, so that $T(W_0)\cap T(W_1)=\emptyset$.  Since $X = \overline{W_1^+}\cup \overline{W_1^-}$, we have $T(W_0) \subseteq \overline{W_1^+}\cup \overline{W_1^-}$.  On the other hand $T(W_0)$ does not meet $T(W_1) = \overline{W_1^+}\cap \overline{W_1^-}$.  It is therefore contained in a connected component of $X \setminus T(W_1)$.  Each such connected component is contained either in the full subcomplex on $W_1^+$ or the full subcomplex on $W_1^-$.  In particular we have $T(W_0)^{(0)}\subseteq W_1^+$ or $T(W_0)^{(0)}\subseteq W_1^-$.  

  Arguing similarly for $T(W_1)$, there are  $\epsilon_0,\epsilon_1\in \{\pm\}$ so that
  \[ T(W_0)^{(0)}\subseteq W_1^{\epsilon_1},\ T(W_1)^{(0)}\subseteq W_0^{\epsilon_0},\]
  but this implies $W_1^{-\epsilon_1}\cap W_0^{-\epsilon_0}=\emptyset$.  In other words $W_0$ and $W_1$ do not cross.

  For the inductive step suppose that $W_0,\ldots,W_m$ pairwise cross, and let $T_i=T(W_i)$ for each $i$.  The base case implies that $T_0\cap T_m\ne \emptyset$, whereas the inductive hypothesis implies that
  \begin{align*} I_1 & := N_{\Theta_{m-1}}(T_0)\cap\cdots\cap N_{\Theta_{m-1}}(T_{m-1})\mbox{ and } \\
    I_2 & :=N_{\Theta_{m-1}}(T_1)\cap\cdots\cap N_{\Theta_{m-1}}(T_{m})
  \end{align*}
  are nonempty.  Take $a\in I_1$, $b\in I_2$, and $c\in T_0\cap T_m$, and consider the geodesic triangle $\Delta$ with corners $a,b,c$.  Let $L_m = \Ftwo(Q+\Theta_{m-1})$.  We may assume without loss of generality that $L_m\ge \delta$, where $\delta$ is the relatively thin triangles constant for $(X,\mc{B})$.
  We argue that we can take $\Theta_m = Q+ \Theta_{m-1}+ 2L_m+ 2\delta$.

  Suppose first that the $\delta$--fat part of $\Delta$ contains a segment of length at least $L_m$ in each side.  Then in particular there is some peripheral complex $Z$ so that $\Delta$ is $\delta$--thin relative to $Z$.   Lemma~\ref{lem:attractivecarrier} applied to the side $[a,c]$ implies that $Z$ is contained in $T_0$; the same lemma applied to $[b,c]$ implies that $Z$ is contained in $T_m$; and the same lemma applied to $[a,b]$ implies $Z$ is contained in each of $T_1,\ldots,T_{m-1}$.  In particular $Z\subseteq \bigcap_{i=0}^mT_i$.

  If on the other hand some $\delta$--fat part of a side of $\Delta$ has length bounded above by $L_m$, then Lemma~\ref{lem:fatpart bounds thinness} shows that $\Delta$ is  $(2L_m+2\delta)$--thin.  In particular if $x$ is the point on $[a,b]$ which is in the preimage of the central point of the comparison tripod for $\Delta$, then $x$ is contained in the intersection of the $\Theta_m$--neighborhoods of the thick carriers $T_0,\ldots,T_m$.
\end{proof}

The main result of this subsection is the following.
\begin{proposition}\label{prop:NofU}
  Let $\U$ be a DCC ultrafilter.  There is some $K\ge 0$ so that $\bigcap\limits_{A\in \U}N_K(\overline{A})\ne \emptyset$.
\end{proposition}
\begin{proof}
  We inductively define a certain collection of pairwise crossing walls $W_0,\ldots,W_k$, and show that the coarse intersection of their carriers is coarsely contained in every halfspace of $\U$.
  
  To begin, choose $A_0$ minimal in $\U$ (with respect to inclusion) and let $W_0 = \{A_0,{A_0}^c\}$.  Note that such an $A_0$ exists because $\U$ is DCC.  Suppose now we have chosen $W_0,\ldots,W_{s-1}$, and let $I_s$ be the set of walls which cross $W_i$ for all $i<s$.  If $I_s$ is nonempty, let $A_s$ be a minimal element of $\{\U(W)\mid W\in I_s\}$, and let $W_s = \{A_s,{A_s}^c\}$.  By Corollary~\ref{cor:cocompact}, the dimension of $C(\mc{W})$ is finite.
  For $s$ bigger than this dimension, the set $I_s$ is empty, so the process of choosing the $W_i$ eventually terminates.

  Let $W_0,\ldots,W_k$ be the resulting collection, and let $T_i = T(W_i)$.
  Now we decompose $\U$ as $\U_0\sqcup\cdots\sqcup \U_k$, where 
\begin{equation*}
  \U_0 = \{A\in \U\mid \{A,A^c\}\mbox{ does not cross }W_0\},
\end{equation*}
and for $i>1$,
\begin{equation*}
  \U_i = \{A\in \U\mid \{A,A^c\}\mbox{ crosses }W_0,\ldots,W_{i-1}\mbox{ but does not cross }W_i\}.
\end{equation*}
The minimality assumptions on the $W_i$ imply that if $A\in \U_i$, then $\U(W_i)\subseteq A$.  By Lemma~\ref{lem:nestingcarriers}, this implies that $T_i\subseteq \overline{A}$.  In particular, we have
\begin{equation*}
  \bigcap_{i=1}^k N_{\Theta_k}(T_i)\subseteq \bigcap_{A\in \U}N_{\Theta_k}(\overline{A}).
\end{equation*}
Lemma~\ref{lem:wallsintersect} shows that the left hand side is nonempty, and the Proposition follows, with $K = \Theta_k$.
\end{proof}

\subsection{Classification of stabilizers}

Given a DCC ultrafilter $\U$, fix $K$ so that
\begin{equation*}
  N_K(\U) \defeq \bigcap_{A\in \U} N_K(\overline{A})
  \end{equation*}
  is nonempty.  The existence of such a $K$ is guaranteed by Proposition~\ref{prop:NofU}.  We observe:
  \begin{lemma}
    The stabilizer of $N_K(\U)$ contains the stabilizer of $\U$.
  \end{lemma}
  Since the action of $G$ on $X$ is proper, we immediately obtain:
  \begin{corollary}\label{cor:boundedNU}
    If $N_K(\U)$ is bounded, then $\Stab(\U)$ is finite.
  \end{corollary}
  We thus focus on the case that $N_K(\U)$ is unbounded.  In this case, we will see in Proposition~\ref{prop:unboundedNU} below that there is always at least one peripheral complex contained in $N_K(\U)$, in fact in the intersection $N_0(\U)$ of the carriers themselves.  Although the constant $K$ depends on $\dim(C(\mc{W}))$, and hence is \emph{not} independent of the choice of $\Rfin$ made in Assumption~\ref{ass:R}, this constant will disappear from the discussion once we have found this peripheral complex in $N_0(\U)$.  

  We need a lemma about the geometry of \CAT$(0)$ cube complexes, which is very similar to Remark 3.2 of Sageev--Wise \cite{SageevWise15}.
  \begin{lemma}\label{lem:sageevwise}
    Let $\Xi$ be a $k$--dimensional \CAT$(0)$ cube complex, and let $\theta_k = \arcsin(\frac{1}{\sqrt{k}})$.  Let $d_\Xi$ be the \CAT$(0)$ metric on $\Xi$.
    Suppose that $d_\Xi(x,y)>2 \sqrt{k}$, and let $m$ be the midpoint of $[x,y]$.  Then there is a hyperplane $H$ meeting $[x,y]$ within $\frac{\sqrt{k}}{2}$ of $m$, and making an angle of at least $\theta_k$ with $[x,y]$.  This hyperplane satisfies $d_\Xi(H,\{x,y\})>\frac{d_\Xi(x,y)-\sqrt{k}}{2\sqrt{k}}$.
  \end{lemma}
  \begin{proof}
    The idea here is from Section 3 of \cite{SageevWise15}.  In \cite[Lemma 3.1]{SageevWise15} it is shown that if $\mathbf{v}$ is a nonzero vector in $\R^k$, and $\mc{E}$ is the standard basis of $\R^k$, then there is a subset $\mc{E}'\subseteq \mc{E}$ of cardinality $k-1$ so that every vector in the hyperplane spanned by $\mc{E}'$ makes an angle of at least $\theta_k$ with $\mathbf{v}$.  It is observed in \cite[Remark 3.2]{SageevWise15} that if $\gamma$ is a ray starting at a corner of a cube $\sigma$ and with a nondegenerate initial segment in $\sigma$, then $\gamma$ meets some midcube of $\sigma$ at an angle of at least $\theta_k$.

    The difference in our situation is that $m$ need not be a vertex.  There is however some nondegenerate segment $\sigma = [m,m']$ contained in a cube $\sigma$ of $\Xi$, which we suppose is of minimal dimension to contain such a segment.  Let $\sigma'$ be a maximal length parallel segment beginning at a vertex of $\sigma$.  By \cite[Remark 3.2]{SageevWise15}, there is a midcube of $\sigma$ meeting $\sigma'$ at an angle of at least $\theta_k$.  Let $H$ be the hyperplane determined by this midcube.  We show that $[x,y]$ meets $H$ as in the conclusion of the lemma.

    The carrier $C$ of $H$ can be identified with $H\times [0,1]$, where $H$ is identified with $H\times\{\frac{1}{2}\}$.  Let $\lambda\co [0,T]\to H\times[0,1]$ be a unit speed geodesic with image the intersection of $[x,y]$ with $C$.   The geodesic $\lambda$ projects to constant speed geodesics $\lambda_H\co [0,T]\to H$ and $\lambda_I\co [0,T]\to [0,1]$.  Any geodesic in $C$ makes a well-defined constant angle $\arctan(\dot\lambda_I/\dot\lambda_H)$ 
    with $H$, and this angle is invariant under translation inside a cube.  Since the angle is at least $\theta_k=\arcsin(\frac{1}{\sqrt{k}})$, the length $T$ of $\lambda$ is at most $\sqrt{k}$.  In particular neither $x$ nor $y$ lies in $C$, and $\lambda_I$ is surjective.  Let $t\in [0,T]$ satisfy $\lambda_I(t)=\frac{1}{2}$, so $\lambda(t)$ lies on the hyperplane $H$.  The distance from $\lambda(t)$ to $m$ is at most $\frac{T}{2}\le \frac{\sqrt{k}}{2}$.

    The last assertion is an easy calculation.
  \end{proof}
In the next lemma we return to the setting of our $(\delta,f)$--relatively hyperbolic pair $(X,\mc{B})$.
  \begin{lemma}\label{lem:pushacrosstriangle}(Pushing across triangles)
        Suppose that $a,b,c\in X$, $Z\in \mc{B}$, and that there is a subsegment $[x,y]\subseteq [a,b]\cap Z$ satisfying:
        \begin{enumerate}
        \item The length of $[x,y]$ is at least $2\delta+2\eta+f(\delta)$; and
        \item $[x,y]$ does not cross the internal point of $[a,b]$, in other words $d(a,x)<d(a,y)\le (b,c)_a$.
        \end{enumerate}
    Then there is a subsegment $[x',y']\subseteq [a,c]\cap Z$ with $|d(a,x')-d(a,x)|\le \delta+\eta$ and $|d(a,y')-d(a,y)|\le \delta + \eta+f(\delta)$.
  \end{lemma}
      \begin{figure}[htbp]
        \centering
        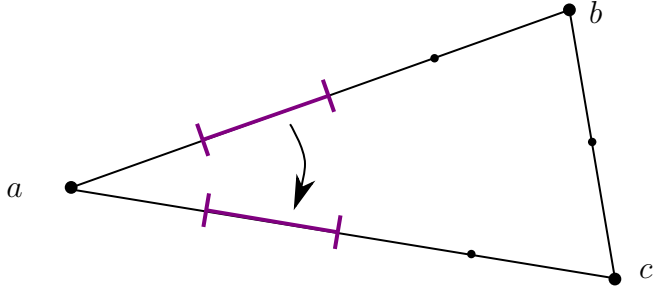
        \caption{Lemma~\ref{lem:pushacrosstriangle}.  The small points are the points in the preimage of the central point of the comparison tripod.}
        \label{fig:pushacrosstriangle}
      \end{figure}
  \begin{proof}
    Let $\Delta$ be the triangle with corners $a,b,c$. Let $\pi\co \Delta\to T_\Delta$ be the canonical map to a comparison tripod $T_\Delta$.  
    Let $I=[x,y]$.  The assumptions on $d(a,x)$ and $d(a,y)$ imply that $\pi(I)$ lies entirely in one leg of $T_\Delta$.  Let $I'$ be the subsegment of $[a,c]$ with the same image in $T_\Delta$.  There are three cases.

    In case $\Delta$ is $\delta$--thin, the segment $I'$ has endpoints $\delta$--close to those of $I$.  In particular it begins and ends within $\delta$ of $Z$, so it intersects $Z$ in a subsegment whose endpoints are at most $\eta+\delta$ from those of $I'$, and we are finished.

    In case $\Delta$ is $\delta$--thin relative to $Z$, then every point of $I'$ is within $\delta$ of some point of $I$ or some point of $Z$.  Since $I\subseteq Z$, the segment $I'$ begins and ends within $\delta$ of $Z$ as before.

    In case $\Delta$ is $\delta$--thin relative to some $W\ne Z$, we note that the $\delta$--fat part of $[a,b]$ can intersect $I$ in a segment of length at most $f(\delta)$.  Removing the $\delta$--fat part thus leaves a subsegment $I_0$.  The corresponding subsegment $I_0'\subseteq I'$ begins and ends within $\delta$ of $Z$, and so we can find a subsegment of $I_0'$ with the desired properties.
  \end{proof}

The main use of the next two lemmas is in the proof of Proposition~\ref{prop:unboundedNU}, which says that whenever $N_K(\U)$ is unbounded $N_0(\U)$ contains a peripheral complex.  Both will be used again later in slightly simpler circumstances.

In the following statement $\omega$ and $\tau$ are the constants which were fixed in Notation~\ref{not:mysterious}; their definitions are recalled in the proof.
\begin{lemma}\label{lem:NJU}
    For any $J\ge 0$, if $x,y\in N_J(\U)$ satisfy $d(x,y)> \left(\frac{2\pi}{\theta_n}+2\right)J + \omega$, then there is a peripheral complex $Z$ which intersects $[x,y]$ in a subsegment $I$ so that
  \begin{enumerate}
  \item the length of $I$ is at least $\min\{\Rfin-\tau, \frac{1}{2}d(x,y)-\tau-J\}$, and
  \item the distance from $I$ to the midpoint of $[x,y]$ is at most $\tau$.
  \end{enumerate}
\end{lemma}
\begin{proof}
 We first remark that it is easy to check that $\left(\frac{2\pi}{\theta_n}+2\right) \ge 2\sqrt{n}+2$.
Recall from Notation~\ref{not:mysterious} the definition of the following two constants:
  \begin{equation*}
    \omega = \defomega,
  \end{equation*}
  and 
  \begin{equation*}
  \tau = \deftau;
    \end{equation*}
  Since $d(x,y)\ge \omega \ge 2\sqrt{n}$, we can apply Lemma~\ref{lem:sageevwise} to find a hyperplane $H$ meeting $[x,y]$ at an angle of at least $\theta_n$, at a point $z\in [x,y]$ satisfying
  \begin{equation}\label{ineq:z}
    d(z,\{x,y\})>\frac{d(x,y)-\sqrt{n}}{2}.    
  \end{equation}
  Since $\left(\frac{2\pi}{\theta_n}+2\right) > 2\sqrt{n}$ and $\omega>\sqrt n$, the last assertion of Lemma~\ref{lem:sageevwise} implies that $H$ misses the $J$--balls around $x$ and $y$:
  \[d(H,\{x,y\})\ge \frac{\left(\frac{2\pi}{\theta_n}+2\right)J + \omega-\sqrt{n}}{2\sqrt{n}} =  \frac{\left(\frac{2\pi}{\theta_n}+2\right)}{2\sqrt{n}} J + \frac{\omega}{2\sqrt{n}}-\frac{1}{2} > J.\]

  Let $W$ be the wall corresponding to $H$, and let $A=\U(W)$.  There are points $x_A,y_A$ in $\overline{A}$ within $J$ of $x$, $y$ respectively.
  \begin{figure}[htbp]
    \centering
    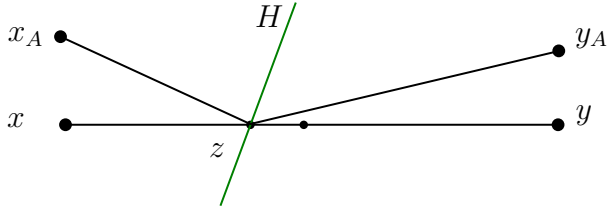
    \caption{The points $x_A$ and $y_A$ are within $J$ of $x$, $y$ respectively, so they are separated by the hyperplane $H$.  The unlabeled point is the midpoint of $[x,y]$, which is distance at most $\sqrt{n}/2$ from $z$.}
    \label{fig:NJU1}
  \end{figure}

  Since $H$ misses the $J$--balls around $x$ and $y$, it must separate $x_A$ from $y_A$.  Observe that $z$ lies in the thick wall carrier $T = T(W)$.
  \begin{claim}\label{claim:zprime}
    Either one of $x_A,y_A$ lies in $T$, or there is a point $z'\in T\cap ([x_A,z]\cup[z,y_A])$ lying in a $W$--hyperplane other than $H$.
   \end{claim}
   \begin{proof}
     If neither $x_A$ nor $y_A$ is in $T$, they both lie in the full sub-complex on $A$.  It follows that $W$ does not separate $x_A$ from $y_A$, and hence there is a second $W$--hyperplane $H'$ separating $x_A$ from $y_A$.  The hyperplane $H'$ meets either $[x,x_A]$ or $[y,y_A]$ at a some point $z'$.
   \end{proof}
   In case one of $x_A,y_A$ lies in $T$, we relabel so that $y_A$ lies in $T$.  If not, we similarly relabel so that the point $z'$ found in Claim~\ref{claim:zprime} lies in the segment $[z,y_A]$.
   \begin{claim}\label{claim:I}
     There is a $W$--peripheral complex $Z$ intersecting $[z,y_A]$ in a segment $I_0$ of length at least
     \begin{multline*}
       \min\left\{\Rfin-\left(\Fone(Q) + 4 Q + 2 \eta\right),\right. \\
        \left. \frac{d(x,y)-\sqrt{n}}{2} - \left(\frac{Q}{\sin(\frac{\theta_n}{2})}+2Q + 2\eta\right) - J \right\},
     \end{multline*}
 and so that $d(I_0,z)\le \frac{Q}{\sin\left(\frac{\theta_n}{2}\right)}+ Q+\eta$.
   \end{claim}
   \begin{proof}
     Note that the angle $\psi$ between $[z,y]$ and $[z,y_A]$ is less than $\frac{\theta_n}{2}$.  Indeed, the hypothesis on $d(x,y)$, and the inequalities \eqref{ineq:z},  $\omega >\sqrt{n}$ and $d(y,y_A)\le J$ together imply (since at any rate $\psi<\frac{\pi}{2}$):
     \begin{equation}
       \psi \le \frac{\pi}{2}\sin(\psi)\le \frac{\pi}{2}\left(\frac{d(y,y_A)}{\min\{d(z,y),d(z,y_A)\}}\right) < \frac{\pi}{2}\left(\frac{J}{\pi J/\theta_n}\right) = \frac{\theta_n}{2}.
     \end{equation}
     (Here we are using the fact that angles in a \CAT$(0)$ triangle are dominated by the angles in the comparison triangle.)
     Since the angle $[z,y]$ makes with $H$ is at least $\theta_n$, the angle $[z,y_A]$ makes with $H$ is at least $\theta_n-\psi >\frac{\theta_n}{2}$.  It follows that the subsegment $[z,p]$ of $[z,y_A]$ lying in $N_Q(H)$ has length at most
     $\frac{Q}{\sin\left(\frac{\theta_n}{2}\right)}$.  
       
     The point $p$ just defined must lie in the $Q$--neighborhood of some $W$--peripheral complex $Z$.  This is the peripheral complex in which we will find the segment $I_0$.

     In case all of $[p,y_A]$ lies in $N_Q(Z)$, we use $\eta$--super-attractiveness to obtain a segment $I_0\subseteq [p,y_A]\cap Z$ so that
     \[ |I_0| \geq d(p,y_A) - (2 Q + 2 \eta) \geq \frac{d(x,y)-\sqrt{n}}{2} - \left(\frac{Q}{\sin(\frac{\theta_n}{2})}+2Q + 2\eta\right) - J,\]
     as desired.

     Next we suppose that not all of $[p,y_A]$ lies in $N_Q(Z)$.  It follows that $y_A$ does not lie in $Z$, since the $Q$--neighborhood of $Z$ is convex.  
      Claim \ref{claim:zprime} implies that either $y_A$ or some point $z' \in [p,y_A]$ lies on a $W$--hyperplane or $W$--peripheral complex not equal to $Z$, so either $[z,y_A]$ or $[z,z']$ is contained in $N_Q(T)$ but not entirely contained in $N_Q(Z)$.  Let $q$ be the first point on $[p,y_A]$ within $Q$ of some $W$--hyperplane $H''\neq H$ or some $W$--peripheral complex $Z'\neq Z$.  Note the point $q$ must also lie in $N_Q(Z)$.   In the case $q$ is within $Q$ of a $W$--peripheral complex $Z'\neq Z$, there exists a hyperplane $H''$ so that $d(q,H'') \leq \Fone(Q)$.  In either case there is a hyperplane $H'' \neq H$ so that $d(q,H'') \leq \max\{Q,\Fone(Q)\}$, and so we have
     \[ d(p,q) \geq \Rfin - (Q + \max\{Q,\Fone(Q)\}) \geq \Rfin  - (\Fone(Q) + 2 Q).\]
     
     The $\eta$--super-attractiveness gives us a subsegment $I_0$ of $[p,q]\cap Z$ of length at least
     \[ d(p,q) - (2Q + 2\eta) \geq \Rfin  - ( \Fone(Q) + 4 Q + 2 \eta ).\]

     In either case $d(I_0,z)\le d(p,z) + Q + \eta \le \frac{Q}{\sin\left(\frac{\theta_n}{2}\right)}+ Q+\eta$, as desired.
    \end{proof}
    We now wish to apply Lemma~\ref{lem:pushacrosstriangle} (Pushing across triangles) to the segment $I_0$ from the last claim, but it may extend past the internal point of the side $[z,y_A]$ of the triangle with corners $\{z,y,y_A\}$.  We note however that $\gprod{y}{y_A}{z}\ge \frac{d(x,y)-\sqrt{n}}{2}-J$ and $d(z,I_0)\le \frac{Q}{\sin\left(\frac{\theta_n}{2}\right)}+ Q+\eta$ by Claim~\ref{claim:I}, so $I_0$ contains a subsegment $I_1$ as in the hypothesis of that lemma with
  \begin{equation*}
    |I_1|
    \ge\min\left\{|I_0|,\frac{d(x,y)-\sqrt{n}}{2}- \left(\frac{Q}{\sin\left(\frac{\theta_n}{2}\right)}+ Q+\eta\right)-J\right\}.
  \end{equation*}
  This is a better lower bound than that already given for $I_0$ in the statement of Claim~\ref{claim:I}, so we may assume that $I_0$ already satisfies the hypotheses of Lemma~\ref{lem:pushacrosstriangle}.  Pushing across the triangle we obtain a subsegment $I$ in $[z,y]\subseteq [x,y]$ which is contained in $Z$ and has length at least $|I_0| - (2\delta + 2\eta + f(\delta))$.  Moreover the distance from $I$ to $z$ is no more than $d(I_0,z) + \eta + \delta$.

Adding up the constants we obtain the inequalities asserted in the lemma.
\setcounter{claim}{0}    
\end{proof}

  \begin{lemma}\label{lem:pcinN}
    Let $J\ge 0$, and $x,y\in N_J(\U)$.  Suppose that $Z$ is a peripheral complex so that $Z\cap[x,y]$ contains a segment $\sigma$ so that
      $\mathrm{length}(\sigma) > \Ftwo(0)+4(\eta+\delta)+2f(\delta)$, and 
      $d(\sigma,\{x,y\}) > J+\delta+\eta$. 
    Then $Z\subseteq N_0(\U)$.
  \end{lemma}
  \begin{proof}
    Let $B\in \U$.  There are points $x_B,y_B$ of $\overline{B}$ within $J$ of $x,y$ respectively.  We claim that there is a subsegment $\sigma'$ of $[x_B,y_B]\cap Z$ of length at least $\Ftwo(0)$.  Lemma~\ref{lem:attractive halfspace carriers} then implies that $Z\subseteq \overline{B}$.  Since $B\in \U$ was arbitrary, we obtain $Z\subseteq \bigcap_{A\in \U}\overline{A} = N_0(\U)$.

    Finding the subsegment $\sigma'$ is a matter of applying Lemma~\ref{lem:pushacrosstriangle} (Pushing across triangles) twice; first the Lemma is applied to a triangle whose corners are $x,y,y_B$, and then to a triangle whose corners are $x,x_B,y_B$.  The lower bound on $d(\sigma,\{x,y\})$ implies the hypothesis about Gromov products in Lemma~\ref{lem:pushacrosstriangle}.  Each time we push across a triangle, we can lose only up to $2(\eta+\delta) + f(\delta)$ in length, so we are left with at least $\Ftwo(0)$ in the end.
  \end{proof}

  \begin{proposition}\label{prop:unboundedNU}
    If $N_K(\U)$ is unbounded then there is a peripheral complex contained in
    \begin{equation*}
      N_0(\U)\defeq \bigcap_{A\in\U}\overline{A}.
    \end{equation*}
  \end{proposition}
  \begin{proof}
    Assuming $N_K(\U)$ is unbounded, we may choose two vertices $x,y\in N_K(\U)$ so $d(x,y)$ is as large as we like.  In particular we make the following assumption:
    \begin{equation}
      \label{eq:dxylowerbound}
      d(x,y)> \max\left\{\left(\frac{2\pi}{\theta_n}+2\right)K+\omega,2\Rfin+2K+2\delta+2\eta \right\}.
    \end{equation}
    In particular, we have $\frac{1}{2}d(x,y)>\Rfin+K$, so Lemma~\ref{lem:NJU} with $J=K$ gives us a peripheral complex $Z$ and a segment $I\subseteq [x,y]$ so that
    \begin{enumerate}
    \item the length of $I$ is at least $\Rfin-\tau$, and
    \item the distance from $I$ to the midpoint of $[x,y]$ is at most $\tau$.
    \end{enumerate}
    Note also that we have the following inequalities. (The first line follows from Assumption~\ref{ass:R}.\eqref{req:for unboundedNU}; the second from~\eqref{eq:dxylowerbound}.)
    \begin{align*}
      |I| \ge \Rfin - \tau & > \Ftwo(0) + 4(\eta + \delta) + 2f(\delta)\\
      \frac{1}{2}d(x,y) & >     \Rfin + (K+\delta + \eta).
    \end{align*}
    Since $d(x,y)>2(\Rfin+K+\delta+\eta)$ we can choose such an $I$ to satisfy $d(I,\{x,y\})>K+\delta+\eta$.  Applying Lemma~\ref{lem:pcinN}, we see $Z\subseteq N_0(\U)$.
  \end{proof}

  \begin{proposition}\label{prop:parabolicsinstabU}
    Let $Z$ be a peripheral complex.  The following are equivalent:
    \begin{enumerate}
      \item\label{inN0} $Z\subseteq N_0(\U)$.
      \item\label{fi} $\Stab(\U)$ contains a finite index subgroup of $\Stab(Z)$.
      \item\label{inf} $\Stab(\U)$ contains an infinite subgroup of $\Stab(Z)$.  
    \end{enumerate}
  \end{proposition}
  \begin{proof}
   We first prove \eqref{inN0}$\implies$\eqref{fi}.    The hypothesis that $Z\subseteq N_0(\U)$ implies that for every wall $W$, either $Z^{(0)}\subseteq \U(W)$ or $Z$ is a $W$--peripheral complex.  Let $\U'\subseteq\U$ consist of those halfspaces $A\in \U$ so that $Z^{(0)}\subseteq A$.  Note that every element of $\Stab(Z)$ preserves $\U'$.  If $\U(W)\notin \U'$, then $Z$ is a $W$--peripheral complex.  There are only finitely many such $W$, so $\U\setminus \U'$ is finite.  Let $S=(\U\setminus\U')\cup\{A^c\mid A\in \U\setminus\U'\}$.  The finite set $S$ is preserved by $\Stab(Z)$ and
the kernel of $\Stab(Z)\to \operatorname{Sym}(S)$ lies in $\Stab(\U)$.

  The implication \eqref{fi}$\implies$\eqref{inf} is obvious, since the peripheral subgroups of $(G,\mc{P})$ are assumed to be infinite (recall Convention~\ref{conv:P_infinite}).

  Finally we show \eqref{inf}$\implies$\eqref{inN0}.  Let $Z$ be a peripheral complex so that $\Stab(Z)\cap\Stab(\U)$ is infinite.  Let $x$ be any point of $N_0(\U)$.  Choose $g\in \Stab(Z)\cap\Stab(\U)$ so that $d(x,gx)= E \ge \Ftwo(Q)+2(d(x,Z)+\eta)$.

  Let $A = \U(W)$ where $W$ is a wall.  Both $x$ and $gx$ lie in $\overline{A}$, which is $Q$--quasi-convex, so $[x,gx]$ lies in $N_Q(\overline{A})$.  The super-attractiveness of $Z$ ensures that a subsegment of $[x,gx]$ of length at least $E-2(d(x,Z)+\eta)\ge \Ftwo(Q)$ lies in $Z$.  The endpoints are distance at most $Q$ from $\overline{A}$, so Lemma~\ref{lem:attractive halfspace carriers} implies that $Z\subseteq \overline{A}$.  Since $A\in \U$ was arbitrary, $Z\subseteq N_0(\U)$.
  \end{proof}
  
\subsection{A $\Stab(\U)$--tree}
For the rest of this section we focus on the case that $\Stab(\U)$ is infinite and not parabolic.  By Corollary~\ref{cor:boundedNU} and Proposition~\ref{prop:unboundedNU}, there must be at least one peripheral complex contained in $N_0(\U)$.  If there is only one, then its stabilizer obviously contains $\Stab(\U)$, so we can assume there is more than one peripheral complex in $N_0(\U)$.  We consider projections of these peripheral complexes to each other.  These have diameter bounded by $C$ (see Notation~\ref{not:understandable} and Corollary~\ref{cor:bounded_proj}).  For any pair $Z,Z'$ of peripheral complexes in $N_0(\U)$, the next lemma shows that the ultrafilter $\U$ is ``weakly principal'' for any point in $\pi_Z(Z')$.
\begin{lemma}\label{lem:weaklyprincipal}
  For $\kappa$ as in Notation~\ref{not:mysterious}, the following holds:  If $Z,Z'$ are peripheral complexes contained in $N_0(\U)$, then 
    \begin{equation*}
    \pi_Z(Z')\subseteq \bigcap_{A\in\U}N_{\kappa}(A).
  \end{equation*}
\end{lemma}
\begin{proof}
Recall $\kappa = \defkappa$, where $\Fone$ is the function from Lemma~\ref{lem:closetohyperplane}, and $C$ is the constant from Corollary~\ref{cor:bounded_proj}.
Fix $A \in \U$ and suppose that $A = \U(W)$.   Or goal is to show that $\pi_Z(Z') \subseteq N_\kappa(A)$.

First suppose that $Z$ is not a $W$--peripheral complex.  In this case, since $Z \subseteq N_0(\U) \subseteq N_0(\overline{A})$ we see that $Z^{(0)} \subseteq A$, so $Z \subseteq N_\kappa(A)$.  Since $\pi_Z(Z') \subseteq Z$ we are done in this case.

Therefore, we suppose henceforth that $Z$ is a $W$--peripheral complex, so in particular $Z \subseteq T(W)$.  Let $b \in Z'$ be arbitrary, and consider the geodesic $[b,\pi_Z(b)]$.

We remark that peripheral complexes intersect, so it is possible that $b = \pi_Z(b)$.
\begin{case}
  $b = \pi_Z(b)$.
\end{case}
If $Z'$ is also a $W$--peripheral complex, Lemma~\ref{lem:closetohyperplane} implies that $\pi_Z(b)$ is within $\Fone(0)\le \Fone(Q+1)<\kappa$ of some $T(W)$--hyperplane.  It follows that in this case $\pi_Z(Z') \subseteq N_\kappa(A)$, as required.

If $Z'$ is not a $W$--peripheral complex, then we have $(Z')^{(0)}\subseteq A$, and hence $b \in N_1(A)$.

We suppose for the remainder of the proof that $b \ne \pi_Z(b)$, and so $[b,\pi_Z(b)]$ is a non-degenerate geodesic which only intersects $Z$ at $\pi_Z(b)$.
\begin{case}
  $b\neq \pi_Z(b)$ and  $[b,\pi_Z(b)] \cap T(W) \ne \{ \pi_Z(b) \}$.
\end{case}

In this case, let $x \in [b,\pi_Z(b)] \cap (T(W) \smallsetminus Z)$.  Since $Z \subseteq T(W)$ and $T(W)$ is $Q$--quasi-convex (Notation~\ref{not:understandable}), the geodesic $[x,\pi_Z(b)]$ lies in the $Q$--neighborhood of $T(W)$.  There are two subcases:
\begin{subcase}
  $d(x,\pi_Z(b))\ge Q+1$.
\end{subcase}

In this subcase let $y$ be the point on $[x,\pi_Z(b)]$ at distance $Q+1$ from $\pi_Z(b)$.  Since $\pi_Z(b)$ is the closest point on $Z$ to $b$, the point $y$ is not within $Q$ of $Z$, so it is within $Q$ of some $W$--hyperplane or some other $W$--peripheral complex $Z_1$.  If it is within $Q$ of a $W$--hyperplane, we are done, since $Q<\kappa$.  Otherwise it is within distance $Q+1$ of both $Z$ and $Z_1$, and so it is distance at most $\Fone(Q+1)$ from some $W$--hyperplane (Lemma~\ref{lem:closetohyperplane}).  Since $\Fone(Q+1)<\kappa$ we are finished in this case as well.

\begin{subcase}
  $d(x,\pi_Z(b))< Q+1$
\end{subcase}
In this case the point $x$ itself is either on a hyperplane or on a $W$--peripheral complex $Z_1 \neq Z$, and can argue as in the previous subcase, substituting $x$ for $y$.

\begin{case}
 $b\neq \pi_Z(b)$ and  $[b,\pi_Z(b)] \cap T(W) = \pi_Z(b)$.
\end{case}

In this case, since $(Z')^{(0)} \subseteq A$, it must be that the geodesic $[b,\pi_Z(b)]$ lies entirely in the full sub-complex on $A$, meaning that $\pi_Z(Z')$ lies in $N_\kappa(A)$, as required.
\end{proof}

We want to build a bipartite $\Stab(\U)$--tree in order to analyze the structure of $\Stab(\U)$ and prove Property~\eqref{wrg:stabs:gog} of Definition~\ref{def:wrg}.  The equivalence classes under the relation defined in Definition~\ref{def:reln} will form one color of vertex of this tree.  The other color will be the set of peripheral complexes in $N_0(\U)$.

Recall from Notation~\ref{not:mysterious} that
  \begin{equation*}
    D = \defD.
  \end{equation*}

\begin{definition}\label{def:reln}
Define a relation on pairs $(Z,Z')$ where $Z\neq Z'$ are contained in $N_0(\U)$ by saying $(Z_1,Z_1')\sim (Z_2,Z_2')$ if for every geodesic $\gamma$ joining $\pi_{Z_1}(Z_1')$ to $\pi_{Z_2}(Z_2')$, and every peripheral complex $Z$, we have $\diam(Z\cap\gamma)< D$.
\end{definition}

\begin{proposition} \label{prop:eq rel}
  The relation defined in Definition~\ref{def:reln} is an equivalence relation.
\end{proposition}
\begin{proof}
Symmetry is obvious, and reflexivity follows from Corollary~\ref{cor:bounded_proj} because $D > C$ (the upper bound on the diameter of $\pi_Z(Z')$).

  To show transitivity, we argue by contradiction.  Suppose $(Z_1,Z_1')\sim(Z_2,Z_2')\sim(Z_3,Z_3')$, but $(Z_1,Z_1')\not\sim(Z_3,Z_3')$.  For each $i$, let $S_i=\pi_{Z_i}(Z_i')$.  Since $(Z_1,Z_1')\not\sim(Z_3,Z_3')$, there is a geodesic $\gamma$ joining $S_1$ to $S_3$ and a peripheral complex $Z_0$ so that $\diam(\gamma\cap Z_0)\ge D$.  Denote the endpoints of $\gamma$ by $x_1,x_3$, so $x_i \in S_i$.  We first show that for either this or some other peripheral complex, the diameter must be much larger.
  \begin{claim*}
    There exists a peripheral complex $Z$ so that $\diam(\gamma\cap Z)>\frac{3}{4}\Rfin$.    
  \end{claim*}
  \begin{proof}[Proof of Claim]
    Let $m$ be the midpoint of $\gamma \cap Z_0$, and let $x_1',x_3'$ be the endpoints of $\gamma\cap Z_0$.  By Lemma~\ref{lem:sageevwise} there exists a hyperplane $H$ cutting $\gamma$ at a point $m_0$ within $\sqrt{n}$ of $m$ so that \[ d(H,\{ x_1, x_3 \}) >d(H,\{x_1',x_3'\}) > \frac{d(x_1',x_3') - \sqrt{n}}{2\sqrt{n}} > \kappa.\]  (The last inequality holds because $d(x_1',x_3')\ge D > 2\sqrt{n}\kappa + \sqrt{n}$.)
Let $W$ be the wall determined by $H$.  Since $m_0 \in \gamma \cap Z_0$ it is clear that $Z_0$ is a $W$--peripheral complex.

We claim some $W$--hyperplane $H'\ne H$ comes within $\kappa$ of $\gamma$.  Indeed, if no such hyperplane crosses $\gamma$, then one endpoint of $\gamma$ is not in the full subcomplex on $\U(W)$.  From Lemma~\ref{lem:weaklyprincipal} we see that the endpoint must nonetheless lie in $N_{\kappa}(\U(W))$, so it must lie within $\kappa$ of some $W$--hyperplane $H'$.

Let $y \in H'$ and $y_0 \in \gamma$ satisfy $d(y,y_0) \le \kappa$.  Consider the geodesic $\gamma' = [y,m_0]$.  Since $y$ and $m_0$ lie in distinct $W$--hyperplanes and $T(W)$ is $Q$--quasi-convex, $[y,m_0] \subseteq N_Q(T(W))$.  Moreover, distinct hyperplanes are distance at least $\Rfin$ apart, so there must be a segment of $[a,b]$ of length at least $\Rfin-2\Fone(Q)$ which lies outside the $\Fone(Q)$--neighborhood of all the $W$--hyperplanes.  
By Lemma~\ref{lem:closetohyperplane}, this subsegment must lie in the $Q$--neighborhood of a single $W$--peripheral complex $Z$.  By $\eta$--super-attractiveness of peripheral complexes, there is a subsegment $I$ of $[y,m_0]$ of length at least $\Rfin - 2\Fone(Q) - 2(Q+\eta)$ contained in $Z$.  Now consider the geodesic triangle with vertices $y,y_0,m_0$.  Since $d(y,y_0) \le \kappa$ we have $(y,y_0)_{m_0} \ge |\gamma| -  \kappa$.  By adjusting an endpoint of $I$ by at most $\kappa$ and labeling the endpoints of $I$ appropriately by $u,v$, we can ensure the inequalities $d(m_0,u) < d(m_0,v) \le (y,y_0)_{m_0}$.  We still have a lower bound on the length $|I| \ge \Rfin - 2\Fone(Q) - 2(Q+\eta) -\kappa$.  Applying Lemma~\ref{lem:pushacrosstriangle} (Pushing across triangles) to the adjusted $I$, there is a subsegment $J$ of $[y_0,m_0]$ contained in $Z$ of length at least $|I| - ( 2 \delta + 2 \eta + f(\delta)) \ge \Rfin-2\Fone(Q)-2Q - 4\eta - 2\delta-f(\delta)$.
Because of Assumption~\ref{ass:R}.\eqref{req:for prop:eq rel}, this implies $|J| > \frac{3}{4}\Rfin$, as required.
  \end{proof}
  Now consider a point $x_2 \in S_2$ and consider the geodesic triangle with vertices $[x_1,x_2,x_3]$.  Since the side $[x_1,x_3]$ has a segment contained in $Z$ of length at least $\frac{3}{4}\Rfin$ there is a subsegment of length at least $\frac{3}{8}\Rfin$ not crossing the internal point.  Lemma~\ref{lem:pushacrosstriangle} can be applied to this segment to obtain a subsegment of either $[x_1,x_2]$ or $[x_2,x_3]$ which lies in $Z$ and has length at least $\frac{3}{8}\Rfin - \left(2\eta + 2\delta + f(\delta) \right)$.  By Assumption~\ref{ass:R}.\eqref{req:for prop:eq rel.2} the segment has length greater than $D$.  This contradicts $(Z_1,Z_1')\sim(Z_2,Z_2')\sim(Z_3,Z_3')$, completing the proof of Proposition~\ref{prop:eq rel}.
\end{proof}

\begin{lemma}\label{lem:equivclassboundedset}
If $Z,Z'\subseteq N_0(\U)$, then
  \[ \diam\left( \bigcup\left\{\pi_{Z_1}(Z_1')\, \left|\,  (Z_1,Z_1')\sim(Z,Z')\right.\right\}\right)<\max \left\{ \omega, 2(D+\tau) \right\}.\]
\end{lemma}
\begin{proof}
  Recall $\omega$ and $\tau$ are the constants defined in Notation~\ref{not:mysterious} and used in Lemma~\ref{lem:NJU}.

Suppose that $Z,Z',Z_1$ and $Z_1'$ are peripheral complexes in $N_0(\U)$, and that $x \in \pi_Z(Z')$ and $y \in \pi_{Z_1}(Z_1')$ satisfy $d(x,y) \ge \max \left\{ \omega, 2(D+\tau) \right\}$.  Our goal is to show $(Z,Z') \not\sim (Z_1,Z_1')$.  

By Lemma~\ref{lem:NJU} (with $J=0$), since $d(x,y) \ge \omega$ and $x,y \in N_0(\U)$, there is a peripheral complex $Z_0$ which intersects $[x,y]$ in a segment $I$ of length at least $\min \{ \Rfin-\tau, \frac{1}{2}d(x,y)-\tau \}$.  Since $\Rfin > D + \tau$ because of Assumption~\ref{ass:R}.\eqref{req:for equivclassboundedset},
we see that $|I| \ge D$.  It follows from the definition of the relation that $(Z,Z') \not\sim (Z_1,Z_1')$, as required.
\end{proof}

\begin{definition}[The graph $\Lambda$]
  Let $V_0$ be the set of peripheral complexes contained in $N_0(\U)$, and let $V_1$ be the set of equivalence classes $[Z,Z']$ where both $Z,Z'$ are contained in $N_0(\U)$.  To form the graph $\Lambda$, connect a vertex $Z$ of $V_0$ to every vertex of the form $[Z,Z']$ by an edge.
\end{definition}
\begin{lemma}
  The graph $\Lambda$ is connected.
\end{lemma}
\begin{proof}
  We argue by contradiction.  If $\Lambda$ is disconnected, then there are vertices corresponding to peripheral complexes $Z_1,Z_2$ which are in separate components.  Choose such a pair so that $d(Z_1,Z_2)$ is minimal among such pairs.

  Since $Z_1$ and $Z_2$ are in different components we must have $(Z_1,Z_2)\not\sim(Z_2,Z_1)$.  Let $\gamma$ be a geodesic joining $\pi_{Z_1}(Z_2)$ to $\pi_{Z_2}(Z_1)$ and containing a subsegment $\sigma_0$ of length at least $D$ in some peripheral complex $Z$.
  We claim that $Z$ cannot be equal to either $Z_1$ or $Z_2$.  Indeed, let $\rho$ be a shortest geodesic from $Z_1$ to $Z_2$, so the endpoints of $\rho$ lie in $\pi_{Z_1}(Z_2)$ and $\pi_{Z_2}(Z_1)$, respectively, and note that $\rho$ intersects $Z_1$ only at its initial point, and $Z_2$ only at its terminal point.  Since $\diam\left(\pi_{Z_1}(Z_2)\right), \diam\left( \pi_{Z_2}(Z_1) \right) \le C$ (see Notation \ref{not:understandable}), it follows that 
  $$ | \rho | \le | \gamma | \le | \rho | + 2C ,		$$
  and that $\gamma$ spends no more than $C$ in either $Z_1$ or $Z_2$.  
  Since $D>C$ (see Notation \ref{not:understandable}), $Z$ cannot be equal to either $Z_1$ or $Z_2$.
Since $D>\Ftwo(0)+4(\eta+\delta)+2f(\delta)+2(\eta+\delta)$, there is a subsegment $\sigma$ of $\sigma_0$ of length at least $\Ftwo(0)+4(\eta+\delta)+2f(\delta)$ at distance at least $\eta+\delta$ from $\{x,y\}$.  We can therefore apply Lemma~\ref{lem:pcinN} with $K=0$ to conclude that $Z\subseteq N_0(\U)$.

We claim that $d(Z,Z_1)$ is smaller than $d(Z_1,Z_2)$.  Indeed, if $\xi$ is a geodesic between $Z_1$ and $Z_2$ realizing $d(Z_1,Z_2)$ then $\xi$ starts in $\pi_{Z_1}(Z_2)$ and ends in $\pi_{Z_2}(Z_1)$.  It follows from Corollary~\ref{cor:bounded_proj} that $|\gamma| \le |\xi| + 2C$.  On the other hand, the distance from $Z_1$ to $Z$ is at most $|\gamma| - D \le |\xi| + 2C-D < |\xi| = d(Z,Z_1)$ (recall from Notation~\ref{not:mysterious} that $D \le 2C + 1$).  Therefore, by the minimality of $d(Z_1,Z_2)$, $Z$ and $Z_1$ lie in the same component of $\Lambda$.  Similarly, $d(Z,Z_2) < d(Z_1,Z_2)$, and so $Z$ and $Z_2$ lie in the same component of $\Lambda$, a contradiction.
\end{proof}
\begin{lemma}\label{lem:Lambda_tree}
  The graph $\Lambda$ is a bipartite tree.
\end{lemma}
\begin{proof}
That $\Lambda$ is bipartite follows immediately from the definition.  We now prove that $\Lambda$ is a tree.

In order to obtain a contradiction, suppose that $\sigma$ is an oriented cycle of length $2n$ in $\Lambda$.  The vertices of $\sigma$ are labeled (with subscripts to be understood modulo $n$:
\begin{equation*}
  Z_0,\enskip [Z_0,Z_0'']=[Z_1,Z_1'],\enskip Z_1, \enskip [Z_1,Z_1'']=[Z_2,Z_2'],\ldots,[Z_n,Z_n'']=[Z_0,Z_0'].
\end{equation*}
For each $i$ let $b_i$ be a geodesic segment in $Z_i$ from $p_i\in \pi_{Z_i}(Z_i')$ to $q_i\in \pi_{Z_i}(Z_i'')$, and let $a_i$ be the geodesic from $q_i$ to $p_{i+1}$.
Let $\gamma$ be the broken geodesic loop $b_0a_0\cdots a_{n-1}$.
By Lemma~\ref{lem:equivclassboundedset} we have $|a_i| < \max \left\{ \omega, 2(D+\tau) \right\}$.

We claim that the path $\gamma$ satisfies the hypothesis of \cite[Theorem 5.6]{Einstein-Hierarchies}, and is therefore a good quality quasi-geodesic (and hence not a loop).    In order to make our situation line up with the constants in \cite{Einstein-Hierarchies}, we apply a similar trick as in the proof of Proposition~\ref{prop:qiecarrier}.  Namely we temporarily modify the function $f$.  To that end,
define the function $f'(m) = \max\{ f(m), \eta,\frac{1}{3}\max \left\{ \omega, 2(D+\tau) \right\} \}$, and note that $(X,\mc{B})$ is $(\delta,f')$--relatively hyperbolic.  Let $M = f'(5\delta)$ as in \cite{Einstein-Hierarchies}.  Also note that since the peripheral complexes are $\eta$--super-attractive, they are $2(\eta+m)$--attractive in the sense of Einstein (see Remark~\ref{rem:attractiveness}), and thus satisfy the attractiveness hypothesis of \cite[Hypotheses 5.5]{Einstein-Hierarchies}.  We now check that the six conditions from \cite[Theorem 5.6]{Einstein-Hierarchies} hold.

Certainly the first is satisfied, and the second is satisfied because $\sigma$ is embedded.  The third is satisfied because $(X,\mc{B})$ is $(\delta,f')$--relatively hyperbolic and $M = f'(5\delta)$.  The fourth is satisfied because by the proof of Proposition~\ref{prop:eq rel} between inequivalent projections there is a peripheral path of length at least $\frac{3}{4}\Rfin$, and Assumption~\ref{ass:R}.\eqref{req:for Lambda tree} forces $\frac{3}{4}\Rfin \ge 26M + 250\delta$.  Since $|a_i| < \max \left\{ \omega, 2(D+\tau) \right\}$, and by the choice of $f'$, we have $|a_i| \le 3M + 63\delta$, which trivially imply that the fifth and sixth conditions from \cite[Theorem 5.6]{Einstein-Hierarchies} hold.

It follows from the conclusion of \cite[Theorem 5.6]{Einstein-Hierarchies} that a geodesic joining the endpoints of $\gamma$ must be long, and in particular $\gamma$ is not a loop, contrary to hypothesis.  It follows that $\Lambda$ is a tree.
\end{proof}
\subsection{Infinite non-parabolic cell stabilizers}
The goal of this subsection is the following:
\begin{theorem}\label{thm:stabilizersaregraphs}
  Suppose that $\Stab(\U)$ is infinite and non-parabolic.  Then $\Stab(\U)$ is a finite graph of groups where the vertex groups are either full parabolic or finite, and the edge groups are finite.
\end{theorem}

For the proofs of Theorems~\ref{thm:stabilizersaregraphs} and~\ref{t:recubulate} we need some facts about various kinds of quasi-convexity in relatively hyperbolic groups.
\begin{definition}\cite{Tran}
  Let $G$ be finitely generated, and let $\Gamma$ be a Cayley graph for $G$.  The subgroup $G_0$ is \emph{strongly quasi-convex} if for every $\lambda\ge 1$, $\epsilon\ge 0$ there is an $r$ so that any $(\lambda,\epsilon)$--quasi-geodesic in $\Gamma$ with endpoints in $G_0$ lies in an $r$--neighborhood of $G_0$.
\end{definition}
Notice that strongly quasi-convex subgroups are always finitely generated and undistorted (see for example \cite[III.$\Gamma$.3.5]{bridhaef:book}).

We deduce the next lemma from theorems of Tran and Hruska.
\begin{lemma}\label{lem:QCinRH}
  Let $(G,\mc{P})$ be relatively hyperbolic, and let $G_0<G$ be full.  The following are equivalent:
  \begin{enumerate}
  \item\label{undist} $G_0$ is finitely generated and undistorted in $G$.
  \item\label{RQC} $G_0$ is relatively quasi-convex in $(G,\mc{P})$.
  \item\label{SQC} $G_0$ is strongly quasi-convex in $G$.
  \end{enumerate}
\end{lemma}
\begin{proof}
  Theorem 1.5 of \cite{HruskaQC} gives \eqref{undist}$\implies$\eqref{RQC} (even without the assumption that $G_0$ is full).  Theorem 1.4 of \cite{HruskaQC} gives \eqref{RQC}$\implies$\eqref{undist} when $G_0$ is full.

  Theorem 1.9 of \cite{Tran} gives \eqref{undist}$\implies$\eqref{SQC} when $G_0$ is full, since finite index and finite subgroups are strongly quasi-convex.  We have already noted that strongly quasi-convex subgroups are finitely generated and undistorted, so \eqref{SQC}$\implies$\eqref{undist}.
\end{proof}

\begin{corollary}\label{cor:FRQ}
  The subgroup $\Stab(\U)$ is full relatively quasi-convex in $(G,\mc{P})$.  In particular it is finitely generated.
\end{corollary}
\begin{proof}
  The hyperplane stabilizers for $G\acts \hatx$ are full relatively quasi-convex (Lemma~\ref{lem:stabhalfspace}) so by Lemma~\ref{lem:QCinRH}, they are strongly quasi-convex.
  By \cite[Theorem 3.26]{Omnibus} the cube stabilizers of $G\acts\hatx$ are also strongly quasi-convex.  In particular $\Stab(\U)$ is strongly quasi-convex in $G$.  It is full by Proposition~\ref{prop:parabolicsinstabU}.  Now apply Lemma~\ref{lem:QCinRH} to deduce that $\Stab(\U)$ is relatively quasi-convex.
\end{proof}

\begin{proof}[Proof of Theorem~\ref{thm:stabilizersaregraphs}]
  The subgroup $\Stab(\U)$ acts on the tree $\Lambda$ described in the last subsection.  A vertex corresponding to a peripheral complex in $N_0(\U)$ has full parabolic stabilizer by Proposition~\ref{prop:parabolicsinstabU}.  The stabilizer of a vertex corresponding to an equivalence class $[(Z,Z')]$ must also stabilize the bounded set from Lemma~\ref{lem:equivclassboundedset}, so it must be finite.  Since every edge is connected to some vertex of this type, edge stabilizers are also finite.

Since $\Stab(\U)$ is finitely generated there exists a $\Stab(\U)$--invariant sub-tree $\Lambda_0$ of $\Lambda$ so that $\leftQ{\Lambda_0}{\Stab(\U)}$ is finite.
\end{proof}

\subsection{Proof of Theorem~\ref{t:recubulate}}
See Definition~\ref{def:wrg} for the definition of a \emph{weakly relatively geometric} action.

\recubulate*

\begin{proof}
We take the cube complex $\hatx$ to be the one dual to the wallspace fixed in Definition~\ref{def:wallspaceW}.

  The cocompactness of $G\acts \hatx$ and ellipticity of peripheral subgroups is Corollary~\ref{cor:cocompact}.

  It suffices to prove the statement about cube stabilizers for vertex stabilizers, since each cube stabilizer is an intersection of vertex stabilizers, and the description passes to intersections.

  So, let $\U$ be an ultrafilter so that $\Stab(\U)$ is infinite.
  Corollary~\ref{cor:FRQ} implies that $\Stab(\U)$ is full relatively quasi-convex.

  Theorem~\ref{thm:stabilizersaregraphs} implies that $\Stab(\U)$ is a graph of finite and full parabolic subgroups as in the conclusion.  
\end{proof}

\section{Proof of Theorem~\ref{t:RH Agol}} \label{sec:Dehn fill}

In this section we prove Theorem~\ref{t:RH Agol}.  The strategy is to use the recubulation given by Theorem~\ref{t:recubulate} and relatively hyperbolic Dehn filling results.   See \cite{osin:peripheral,rhds, agm,VH, Hwide,Omnibus} for more information on relatively hyperbolic Dehn filling.  We first recall the definition.

\begin{definition}[Dehn filling] \label{def:dehnfill}
Suppose that $(G,\mc{P})$ is a group pair, and that $\mc{N} = \{ N_P \unlhd P \mid P \in \mc{P} \}$ is a collection of normal subgroups of the elements of $P$.  The {\em Dehn filling} induced by $\mc{N}$ is
\[	G(\mc{N}) := \rightQ{G}{\left\llangle \bigcup N_P \right\rrangle_G}	.	\]
\end{definition}

\begin{definition}[Wide subgroups and fillings]\label{def:Hwide}
  If $P$ is a group, $B<P$, and $F\subseteq P$ a finite set, then $N\unlhd P$ is \emph{$(B,F)$--wide} if, for all $b\in B$ and $f\in F\setminus B$, the product $bf$ does not lie in $N$.

  Let $H$ be a relatively quasi-convex subgroup of $(G,\mc{P})$, so there is a collection of infinite maximal parabolic subgroups $\mc{D}$ of $H$ so that $(H,\mc{D})$ is relatively hyperbolic.  For each $D\in\mc{D}$ there is some $c_D\in G$ and some $P_D\in \mc{P}$ so that $D^{c_D^{-1}} \le P_D$.

  Let $F\subseteq G\setminus\{1\}$ be finite.  Let $G(\mc{N})$ be a Dehn filling, with $K = \llangle \bigcup\mc{N}\rrangle_G$.  We say the filling $G(\mc{N})$ is \emph{$(H,F)$--wide} if for every $D\in \mc{D}$, the intersection $K\cap P_D$ is $( D^{c_D^{-1}}, F\cap P_D)$--wide.
\end{definition}

\begin{definition}[Sufficiently long and wide Dehn fillings]
  We say that a statement $\mathsf{S}$ holds \emph{for all sufficiently long Dehn fillings} if there is a finite set $F\subseteq G\setminus \{1\}$ so that $\mathsf{S}$ holds for all $G(\mc{N})$ so that $\bigcup\mc{N}$ contains no element of $F$.

  Let $H<G$ be a quasi-convex subgroup.  The statement $\mathsf{S}$ holds \emph{for all sufficiently $H$--wide Dehn fillings} if there is a finite set $F\subseteq G\setminus \{1\}$ so that $\mathsf{S}$ holds for all $G(\mc{N})$ which are $(H,F)$--wide.
\end{definition}

 For the Dehn filling results, we make the following assumption (\emph{weakly relatively geometric} was defined in Definition~\ref{def:wrg}).

\begin{assumption} \label{ass:G is recubulated}
The pair $(G,\mc{P})$ is relatively hyperbolic, and admits a weakly relatively geometric action on a \CAT$(0)$ cube complex $\widehat{X}$.  Further, each element of $\mc{P}$ is residually finite.
\end{assumption}

A key reason for our interest in weakly relatively geometric actions is the following result.

\begin{theorem} \label{t:quotient VS}
    Under Assumption~\ref{ass:G is recubulated}, there exist finite-index subgroups $\left\{ L_P\unlhd P \mid P \in \mc{P} \right\}$ so that the following holds:

  Let $\mc{N} = \{N_P \unlhd P \mid P\in \mc{P}\}$ be chosen so that for each $P\in \mc{P}$,
  \begin{enumerate}
  \item $N_P\le L_P$, and
  \item $P/N_P$ is hyperbolic and virtually special.
  \end{enumerate}
  Then the Dehn filling $G(\mc{N})$ is hyperbolic and virtually special.
\end{theorem}
\begin{proof}
Since the $G$--action on $\widehat{X}$ is cocompact, there are cells $\sigma_1, \ldots, \sigma_k$ in $\widehat{X}$ forming a collection of representatives of the $G$--orbits of cells.  For $1 \le i \le k$ let $Q_i$ be the finite-index subgroup of $\Stab(\sigma_i)$ consisting of elements which fix $\sigma_i$ pointwise, and let $\mc{Q} = \{ Q_1, \ldots , Q_k \}$.  Note that each infinite $Q_i$ is a full relatively quasi-convex subgroup which 
admits a graph of groups decomposition with finite edge groups and finite or full parabolic vertex groups.  By combining \cite[Theorem 1.1]{osin:peripheral}, \cite[Propositions 4.3, 4.4]{agm} and \cite[Corollary 6.5]{Omnibus}, and the assumption that elements of $\mc{P}$ are residually finite, we see that there are finite-index subgroups $\left\{ L_P \unlhd P \mid P \in \mc{P} \right\}$ so that for any $N_P \le L_P$ so $P/N_P$ is hyperbolic and virtually special in $P$, the Dehn filling
\[	\overline{G} = \rightQ{G}{K} := G \left( \{ N_P \mid P \in \mc{P} \} \right)		\]
satisfies
\begin{enumerate}
\item\label{eq:barG hyp} $\overline{G}$ is hyperbolic;
\item\label{eq:cell stab vf} The image of each $Q_i$ is quasi-convex in $\overline{G}$, and splits as a graph of virtually special hyperbolic groups with finite edge groups (and in particular the image of each $Q_i$ is hyperbolic and virtually special); and
\item\label{eq:quot cube} The space $\overline{X} := \leftQ{\widehat{X}}{K}$ is a \CAT$(0)$ cube complex.
\end{enumerate}
Let $\overline{G}$ be any such Dehn filling.  Then $\overline{G}$ acts cocompactly on the \CAT$(0)$ cube complex $\overline{X}$.  By \eqref{eq:barG hyp} $\overline{G}$ is hyperbolic.  By \eqref{eq:cell stab vf}, the cell stabilizers for the $\overline{G}$--action on $\overline{X}$ are quasi-convex and virtually special.  It now follows from \cite[Theorem D]{Omnibus} that $\overline{G}$ is virtually special, as required.
\end{proof}

The following is an immediate consequence of results of Haglund--Wise \cite[Corollary 7.4]{HW08} and Minasyan \cite[Theorem 1.1]{minasyan:subsetgferf}.

\begin{theorem} \label{t:vs qcerf}
Suppose that $\Gamma$ is a hyperbolic virtually special group.  Then every quasi-convex subgroup of $\Gamma$ is separable, and every finite product of quasi-convex subgroups of $\Gamma$ is separable.
\end{theorem}

Given Theorem~\ref{t:recubulate}, the next two results are the technical core of the proof of Theorem~\ref{t:RH Agol}. 

\begin{theorem} \label{t:big fill separable}
Under Assumption~\ref{ass:G is recubulated}, suppose further that $S$ is a relatively quasi-convex subgroup of $(G,\mc{P})$ so that
for all $P \in \mc{P}$ and $g \in G$ the subgroup $P \cap S^g$ is separable in $P$.

Then $S$ is separable in $G$.
\end{theorem}
\begin{proof}
Let $g_0 \in G \smallsetminus S$, and let $\{ L_P \mid P \in \mc{P} \}$ be as in Theorem~\ref{t:quotient VS}.  By \cite[Proposition 4.5, 4.7]{Hwide}, for sufficiently long and $S$--wide fillings of $(G,\mc{P})$, the image of $S$ is quasi-convex and does not contain $g_0$.  However, by \cite[Lemma 5.2]{Hwide}, and the assumption on separability of the subgroups $S^g \cap P$ in $P$, there are such long and $S$--wide fillings.

Taking these fillings to also have filling kernels contained in the $L_P$ as above, and we obtain a virtually special hyperbolic quotient $\pi \co G \to \overline{G}$ so that the image $\overline{S}$ of $S$ is quasi-convex and $\pi(g_0) \not\in \overline{S}$.  By applying Theorem~\ref{t:vs qcerf}, there is a finite quotient
$\lambda \co \overline{G} \to Q$ so that $\lambda(\pi(g_0)) \not\in \lambda(\pi(S))$.  Since $g_0 \in G \smallsetminus S$ was arbitrary, this shows that $S$ is separable in $G$, as required.
\end{proof}

The following result relies on Proposition~\ref{prop:DCS} which is proved in Appendix~\ref{app:DCS}, using techniques from \cite{Hwide}.

\begin{theorem} \label{t:big fill DC sep}
Under Assumption~\ref{ass:G is recubulated}, suppose further that $S_1, S_2$ are relatively quasi-convex subgroups of $(G,\mc{P})$ so that
for all $P \in \mc{P}$ and $g_1, g_2 \in G$ the double coset
$\left( P \cap S_1^{g_1} \right) \left( P \cap S_2^{g_2} \right)$ is separable in $P$.

Then the product $S_1S_2$ is separable in $G$.
\end{theorem}
\begin{proof}
Let $g \in G \smallsetminus S_1S_2$.  Applying Proposition~\ref{prop:DCS} and proceeding similarly to the proof of Theorem~\ref{t:big fill separable}, there exists a Dehn filing $\pi \co G \to \overline{G}$ so that
\begin{enumerate}
\item  $\overline{G}$ is a virtually special hyperbolic group; 
\item $\pi(S_1)$ and $\pi(S_2)$ are quasi-convex in $\overline{G}$; and
\item $\pi(g) \not\in \pi(S_1)\pi(S_2)$.
\end{enumerate}
Since double cosets of quasi-convex subgroups are separable in virtually special hyperbolic groups, as noted in Theorem~\ref{t:vs qcerf}, there is a further finite quotient of $\overline{G}$ which can be used to separate $g$ from $S_1S_2$ in $G$.
\end{proof}

Finally, we restate and prove Theorem~\ref{t:RH Agol}.
\RHAgol*

\begin{proof}[Proof of Theorem~\ref{t:RH Agol}]
That each element of $\mc{P}$ is residually finite is part of the hypothesis of Theorem~\ref{t:RH Agol}, and the rest of Assumption~\ref{ass:G is recubulated} follows from Theorem~\ref{t:recubulate}.

Let $S_1, S_2$ be hyperplane stabilizers for the $G$--action on $X$.  By Lemma~\ref{lem:stabH_RQC} these are relatively quasi-convex.  The separability and double coset separability assumptions in Theorems~\ref{t:big fill separable} and~\ref{t:big fill DC sep} are part of the hypotheses of Theorem~\ref{t:RH Agol}.  Therefore, it follows from Theorem~\ref{t:big fill separable} that $S_1$ is separable in $G$, and from Theorem~\ref{t:big fill DC sep} that $S_1S_2$ is separable in $G$.  It now follows from \cite[Corollary 4.3]{HaglundWise10} that the $G$--action on $X$ is virtually special, as required.
\end{proof}

\section{An example} \label{sec:examples}
In this section we give an example of a group to which Theorem~\ref{t:RH Agol} applies, and for which it is at least not obvious how to apply previous results to deduce virtual specialness.  Note that this example is merely indicative of many examples that could be built using this technique, or others.  That the example is virtually special follows from the following application of our main result.  (By a \emph{standard cubical torus} we mean a unit cube of some dimension with opposite faces identified via translation.)
\begin{proposition}
  Let $Y$ be a $1$--vertex non-positively curved cell complex built from a finite wedge of standard cubical tori
  by attaching finitely many regular right-angled hyperbolic polygons rescaled to have side-length $1$.  Then $Y$ is homeomorphic to a cube complex $C$, and this cube complex is virtually special.
\end{proposition}
\begin{proof}
  We first note that the complex $Y$ is negatively curved away from the tori, any two of which meet only in a point.  It follows that the universal cover of $Y$ is \CAT$(0)$ with isolated flats, and that $\pi_1Y$ is hyperbolic relative to the free abelian subgroups represented by the cubical tori (see \cite{hruska:invariants}).

  Now we describe a cube complex homeomorphic to $Y$, obtained by subdividing the cells of $Y$.  
  The complex $Y$ has a single vertex (the wedge point), so all corners of all the polygons are glued at the wedge point.  The only based loops in a cubical torus of length $1$ are the edges of the standard cellulation of that torus.  To obtain $C$, we subdivide each $n$--torus into $2^n$ cubes of dimension $n$, and each $p$--sided polygon into $p$ squares.  The link at the wedge point is unchanged, but now there are additional vertices.  The link of one of the new vertices in an $n$--torus is a standard $(n-1)$--sphere, together possibly with some arcs of length $\pi$ joining its north and south poles.  (This is if the new vertex lies in the middle of an edge of the standard cellulation which is traversed by some of the polygons.)  The new vertex at the center of a $p$--gon has link which is a circle of length $\frac{p\pi}{2}$.  Since $p\ge 5$, the link condition is satisfied here as well.

  The universal cover of $C$ is therefore a \CAT$(0)$ cube complex on which $\pi_1Y$ acts geometrically.  Since the parabolic subgroups are abelian, the Assumptions~\ref{ass:weak} and~\ref{ass:DCS} both hold, and we may apply our main theorem to conclude that $G$ acts virtually co-specially.
\end{proof}

\begin{example}
  Let $G$ be given by the following presentation:
  \[ \left\langle a_1\ldots,a_6,b_1,\ldots,b_6\left|
        \begin{array}{c} [a_1,b_1],\ldots,[a_6,b_6],\\
          a_1a_2a_3a_4a_5a_6,b_1b_2b_3b_4b_5b_6,\\
          a_4b_2b_5a_3a_1,a_3a_5b_1b_4a_6
        \end{array}
          \right.\right\rangle.\]
  A presentation complex for $G$ can be built from the wedge of six two-dimensional tori by attaching two right-angled hexagons and two right-angled pentagons.  One can check (by hand or with a computer) that the link of the vertex is a graph of girth 4.  This link is highly non-planar, suggesting that the boundary at infinity is most likely non-planar (but see \cite{DHW,BMM} for some cautionary tales).  If the boundary at infinity is non-planar, this gives a proof that $G$ is not virtually a $3$--manifold group, and therefore not covered by previous theorems about $3$--manifolds \cite{PW,WisePUP,VH}.
\end{example}

\section{Application to the Relative Cannon Conjecture} \label{sec:Cannon}
To any relatively hyperbolic pair $(G,\mc{P})$ is associated its \emph{Bowditch boundary} $\partial(G,\mc{P})$ \cite{bowditch:relhyp}.  The Relative Cannon Conjecture asserts that if $\mc{P}$ is a non-empty collection of free abelian groups and $\partial(G,\mc{P})$ is homeomorphic to a $2$--sphere, then $G$ is Kleinian (see \cite[Conjecture 1.3]{GMS} and the discussion in that paper).  The usual Cannon Conjecture makes the same assertion when $\mc{P}$ is empty and $G$ has no non-trivial finite normal subgroup (see \cite[Conjecture 11.34]{Cannon91} and \cite[Conjecture 5.1]{CannonSwenson}).

The following result is a generalization of \cite[Theorem 1.1]{EinsteinG}.  Given the work we have already done, the proof is very similar.  
\begin{theorem} \label{t:wRG RC}
 Suppose that $(G,\mc{P})$ is relatively hyperbolic and $\mc{P}$ is a non-empty collection of free abelian groups.  Suppose further that $\partial(G,\mc{P}) \cong \mathbb{S}^2$, and that $(G,\mc{P})$ acts weakly relatively geometrically on a \CAT$(0)$ cube complex.  Then $G$ is Kleinian.
\end{theorem}
\begin{proof}
  Note that each element $P \in \mc{P}$ fixes a point $\xi_P$ in $\partial(G,\mc{P}) \cong \mathbb{S}^2$ and acts properly cocompactly on the complement \cite{bowditch:relhyp}.  Since this complement is homeomorphic to $\R^2$, each $P\in \mc{P}$ must have rank $2$.
  
  Since $(G,\mc{P})$ acts weakly relatively geometrically on a \CAT$(0)$ cube complex, we may apply Theorem~\ref{t:quotient VS}.  Let $\{ L_P \mid P \in \mc{P} \}$ be the subgroups from the conclusion of that theorem.  For any family of infinite cyclic subgroups $\mc{N}=\{ N_P \mid P \in \mc{P} \}$ with each $N_P \le L_P$ the quotient group $G(\mc{N})$ is a hyperbolic virtually special group (each $P/N_P$ is virtually cyclic, and hence hyperbolic and virtually special).  By \cite[Theorem 1.2]{GMS}, the boundary of $G(\mc{N})$ is a $2$--sphere so long as the filling is sufficiently long.  Since the elements of $\mc{P}$ are free abelian, the parent group $G$ has no nontrivial finite normal subgroup.  By \cite[Theorem 7.2]{rhds}, $G(\mc{N})$ also has no finite normal subgroup (at least for sufficiently long fillings of this form), so $G(\mc{N})$ acts faithfully on its boundary.  Ha\"issinsky~\cite[Theorem 1.10]{Haissinsky} states that a cubulated hyperbolic group with planar boundary is virtually Kleinian.  

  We claim that in fact the quotient $G(\mc{N})$ is Kleinian for sufficiently long such fillings.  We argue as in the last paragraph of Section 1.3 of Markovic~\cite{MarkovicCannon}.
  Ha\"issinsky's theorem gives us a finite index $\Gamma<G(\mc{N})$ so that $\Gamma$ is cocompact Kleinian and hence quasi-isometric to $\Hyp^3$.  This implies $G(\mc{N})$ is quasi-isometric to $\Hyp^3$.  By a result of Cannon--Cooper~\cite{CannonCooper} the group $G(\mc{N})$ acts geometrically on $\Hyp^3$.  Since $G(\mc{N})$ has no finite normal subgroup, it acts faithfully.  In other words it is cocompact Kleinian.

  We now take a sequence of longer and longer fillings of this form, obtaining a collection $\{ G \onto G_i \}$ of quotients so that each $G_i$ is Kleinian.  By the fundamental theorem of relatively hyperbolic Dehn filling~\cite[Theorem 1.1]{osin:peripheral} we may assume that this sequence is \emph{stably faithful} in the sense of~\cite[Definition 10.1]{GMS}.
  Each $G\onto G_i$ gives a representation $\rho_i \co G \to \mathrm{Isom}(\mathbb{H}^3)$.  We now argue as in the proof of~\cite[Corollary 1.4]{GMS} that these representations can be conjugated to a sequence of representations that subconverge to a discrete faithful representation of $G$ into $\operatorname{Isom}(\Hyp^3)$.
\end{proof}

The following is an immediate consequence of Theorems~\ref{t:recubulate} and~\ref{t:wRG RC}.
\begin{corollary}\label{cor:cubeRC}
Suppose that $(G,\mc{P})$ is relatively hyperbolic, that elements of $\mc{P}$ are free abelian, that $\partial(G,\mc{P}) \cong \mathbb{S}^2$, and that $G$ acts properly and cocompactly on a \CAT$(0)$ cube complex.  Then $G$ is Kleinian.
\end{corollary}

\appendix
\section{About double coset separability} \label{app:DCS}

In this appendix we generalize \cite[Proposition 6.2]{Hwide}, a technical result which helps to prove double coset separability.  Most of the proof is the same as in \cite{Hwide}, and we largely keep the same notation from there.

\begin{definition} (cf. \cite[Definition 3.1]{Hwide})
Let $P$ be a group, $B_1, B_2 \subseteq P$ are subgroups, and $S$ a finite subset.  A normal subgroup $N \unlhd P$ is {\em $(B_1,B_2,S)$--wide} if whenever there are $b_1 \in B_1$, $b_2 \in B_2$ and $s \in S$ so that $b_1sb_2 \in N$ we have $s \in B_1B_2$.
\end{definition}
\begin{remark}\label{rem:WideIsWide}
  This is a generalization of wideness as in Definition~\ref{def:Hwide}.  Namely, if $N$ is $(B,B,S)$--wide, then $N$ is $(B,S)$--wide.
\end{remark}

\begin{lemma} \label{lem:Wide lemma} (cf. \cite[Lemma 5.1]{Hwide}) 
Suppose that $P$ is a group and that $B_1, B_2$ are subgroups so that $B_1B_2$ is a separable subset of $P$.  For any finite set $S$ there exists a finite-index normal subgroup $K_S \le P$ so that for any $N \unlhd P$ with $N \le K_S$, the subgroup $N$ is $(B_1,B_2,S)$--wide in $P$.
\end{lemma}
\begin{proof}
  For each $s \in S \smallsetminus B_1B_2$, choose some $P_s \le P$ finite-index so that $B_1B_2 \subseteq P_s$ and $s \not\in P_s$.  Let $P_S = \bigcap \left\{ P_s \mid s \in S \smallsetminus B_1B_2 \right\}$ and let $K_S$ be the normal core of $P_S$.
  The subgroups $K_S \le P_S$ are both finite-index in $P$.  The double coset $B_1B_2$ is contained in  $P_S$.

  Choose $N \unlhd P$ finite-index and contained in $K_S$.  It follows that $N$ is contained in $P_S$, which is all we will use to show $N$ is $(B_1,B_2,S)$--wide.  Let $b_1 \in B_1, b_2 \in B_2$ and $s \in S$ and suppose $b_1sb_2 \in N$.  If $s \in B_1B_2$ then there is nothing to prove so suppose that $s \not\in B_1B_2$.  Then $s \not\in P_S$.  However, $b_1sb_2 \in N \subseteq P_S$, and also $b_1^{-1} \in B_1 \subseteq B_1B_2 \subseteq P_S$ and similarly $b_2^{-1} \in P_S$.  Since $P_S$ is a subgroup, we have $s = b_1^{-1} (b_1 s b_2) b_2^{-1} \in P_S$, which contradicts $s \not\in P_S$.
\end{proof}

The following theorem is a generalization of \cite[Proposition 6.2]{Hwide} and is the same as \cite[Theorem 3.21]{OregonReyes}.  Our proof is similar to the one there.

\begin{proposition} (cf. \cite[Proposition 6.2]{Hwide}) \label{prop:DCS}
Suppose $(G,\mc{P})$ is relatively hyperbolic. 
Let $\mc{H}$ be a finite collection of relatively quasi-convex subgroups of $(G,\mc{P})$, so that for any $H_1, H_2 \in \mc{H}$, any $g_1, g_2 \in G$, and any $P \in \mc{P}$ the double coset $\left( H_1^{g_1} \cap P \right) \left( H_2^{g_2} \cap P \right)$ is separable in $P$.\footnote{In particular (taking $H_1 = H_2$ and $g_1 = g_2$) the subgroups $H_1^{g_1} \cap P$ are separable in $P$.}

Let $F \subseteq G$ be a finite subset.
There exist finite-index subgroups $\left\{ K_P \unlhd P \mid P \in \mc{P} \right\}$ so that if $\mc{N} = \{N_P\unlhd P\mid P\in \mc{P}\}$ is a collection with $N_P\le K_P$ for each $P$, and $K = \llangle \bigcup\mc{N}\rrangle_G$, then for all $f \in F$ and $\Psi, \Theta \in \mc{H}$ for which $1 \not\in \Psi\Theta f$ we have $K \cap \Psi\Theta f = \emptyset$.
\end{proposition}
\begin{proof}
We will follow the proof from \cite[Proposition 6.2]{Hwide}, keeping most of the notation and indicating only the changes which must be made.

We have the following setup from \cite{Hwide}.  It suffices to deal with a single pair $(\Psi,\Theta)$ from $\mc{H}$, for we can then take intersections of the $K_P$ over each of the finitely many pairs to get the required result.  We therefore fix such a pair.

The cusped space for $(G,\mc{P})$ is denoted $X$, and is $\delta$--hyperbolic.  Let $\lambda$ be a quasi-convexity constant for each $H \in \mc{H}$, let $M = \max \{ d_X(1,f) \mid f \in F\}$ and let $\alpha = 10 \delta + 2M + 4\lambda$.  Associated to $\Psi$ is a finite collection $\mc{D}$ of $\Psi$--conjugacy classes of maximal uniquely parabolic subgroups of $\Psi$, and for each $D \in \mc{D}$ there are $P_D \in \mc{P}$ and $c_D \in G$ so $D \le P_D^{c_D}$.  Similarly for $\Theta$ we have a finite collection $\mc{E}$ and for $E \in \mc{E}$ there are $P_E \in \mc{P}$ and $d_E \in G$ so $E \le P_E^{d_E}$.  For $P \in \mc{P}$ choose the set $S_P \subset P$ exactly as in \cite{Hwide}.  Namely, we define
\[ S_P \supseteq \{p\in P\mid d_X(1,p)\leq 44\delta + 8 M  + 16\lambda + 2L_1  + 3\} \]
to be a finite set so that any filling which is 
$(H,S_P)$--wide for every $H\in \mc{H}$ will satisfy the conclusion of \cite[Lemma~4.2]{Hwide} with $L_1,L_2$ as specified in the proof of \cite[Proposition~6.2]{Hwide}.

We make a change from \cite{Hwide} to the requirements on the filling kernels, as we now explain.  In \cite{Hwide} the additional requirements were indexed over certain pairs $(B_1,B_2)$ of parabolic subgroups of $\Psi$ and $\Theta$.  We instead index our requirements over the finite set of triples $(c_D,d_E,p)$, where $P_E = P_D$ and $p \in P_D$ satisfies $d_X(1,p) \le 2\alpha + 4\delta$.    For such a triple $q = (c_D,d_E,p)$ choose a finite index normal subgroup ${K}_{q} \le P_D$ so that any $N \unlhd P_D$ with $N \le {K}_{q}$ is
$\left(D^{c_D^{-1}}, \left(E^{d_E^{-1}} \right)^{p}, S_P \right)$--wide.  This is possible by Lemma~\ref{lem:Wide lemma}.  For each $P \in \mc{P}$ for which $P = P_D$ in one of these triples, let ${K}_P$ be the intersection of all such ${K}_{q}$.

Choose $N_P \le {K}_P$ and consider the filling
\[	G \to G\left( \left\{ N_P \mid P \in \mc{P} \right\} \right) = G/K.	\]

We remark that this filling must be $(H,S_P)$--wide for every $H\in\mc{H}$, by Remark~\ref{rem:WideIsWide}.

There is no issue with the setup for proving $K \cap \Psi\Theta f = \emptyset$.
Moreover, Case 1 just uses $(\Psi,S_P)$--wideness, and works verbatim in this setting.

For Case 2, the beginning of the proof works as written.
The first issue with the proof from \cite{Hwide} in this setting comes with the calculation about $u^{-1}ku$, which in \cite{Hwide} uses the assumption that parabolics are abelian.

We pick up the argument at the reference to Figure 2 in that proof, and adjust it as follows (keeping the notation from \cite{Hwide}):

Note that $k \in K \cap \Stab(A)$, where $A$ is the horoball based on $scP$ and $w \in scP$, so $w^{-1}kw \in N_P \le P$.  Also

\begin{eqnarray*}
w^{-1}kw &=&  (w^{-1}u) (u^{-1}kv) (v^{-1}z) (z^{-1}w)	\\
&=& (w^{-1}u) \left[ (u^{-1}kv) (z^{-1}w) \right] \left( (z^{-1}w)^{-1} (v^{-1}z) (z^{-1}w) \right)	.	
\end{eqnarray*}

Let $B_1 = D^{c^{-1}}$, $B_2 = (E^{d^{-1}})^{(z^{-1}w)^{-1}}$, and $s = (u^{-1}kv)(z^{-1}w)$.  
Note also that $z^{-1}w \in P$, so $B_1, B_2 \le P$.  Moreover, $d_X(1,(z^{-1}w)^{-1}) \le 2\alpha + 4\delta$ (see \cite{Hwide} for details of this claim).  Thus, the triple $(c,d,(z^{-1}w)^{-1})$ is one of the finitely many triples described above, so $N_P$ is $(B_1,B_2,S_P)$--wide.

As in \cite{Hwide}, the $X$--length of $(u^{-1}kv)$ is at most $2\alpha + 2L_1 + 3$ and the $X$--length of $(z^{-1}w)$ is at most $2\alpha + 4\delta$, so $\left[ (u^{-1}kv) (z^{-1}w) \right] \in S_P$.
Let $b_1 = w^{-1}u \in B_1$, and $b_2 = (z^{-1}w)^{-1} (v^{-1}z) (z^{-1}w)$.  The above expression shows
$b_1 \left[ (u^{-1}kv) (z^{-1}w) \right] b_2 \in N_P$.  Because $N_P$ is $(B_1,B_2,S_P)$--wide, we know that $\left[ (u^{-1}kv) (z^{-1}w) \right] \in B_1B_2$, so $w^{-1}kw \in B_1B_2$.  Choose $\beta_1 \in B_1$ and $\beta_2 \in B_2$ so $w^{-1}kw = \beta_1 \beta_2$.  From the definition of $B_2$ there is $e\in E$ so that
 $\beta_2 = w^{-1}z d^{-1} e d z^{-1}w$.

Recall the following equations from \cite{Hwide}:

\[	w = \psi_w c, \ \ z = \psi \theta_z d, \ \ v = \psi \theta_v d	,	\]
where $\psi_w \in \Psi$, $\theta_z , \theta_v \in \Theta$.

Now, similar to \cite{Hwide}, but with a slightly different calculation:

\begin{eqnarray*}
k\psi\theta &=& w (w^{-1}kw) (w^{-1}z) (z^{-1}v) (v^{-1}\psi\theta)	\\
&=& w \left( \beta_1 w^{-1}z d^{-1} e d  z^{-1}w \right) (w^{-1}z) (z^{-1}v) (v^{-1} \psi\theta) \\
&=& w \beta_1 (w^{-1}z)  d^{-1} e d  (z^{-1}v) (v^{-1} \psi\theta) \\
&=& \psi_w (c\beta_1c^{-1}) c(w^{-1}\psi) (\psi^{-1}z)  d^{-1} e d  (z^{-1}v) (v^{-1}\psi\theta) \\
&=& \left[ \psi_w (c\beta_1 c^{-1}) \psi_w^{-1}\psi \right]  \cdot  \left[ \theta_z\, e\, \theta_z^{-1} \theta \right]
\end{eqnarray*}
giving an expression for $k\psi\theta$ as an element of $\Psi\Theta$.
Thus, $k \cdot g = k\psi\theta f \in \Psi\Theta f$, the same contradiction as in \cite{Hwide}.
\end{proof}

\end{document}

%% file: slimquad.pdf_tex
%% Creator: Inkscape inkscape 0.92.2, www.inkscape.org
%% Accompanies image file 'slimquad.pdf' (pdf, eps, ps)
\begingroup%
  \makeatletter%
  \providecommand\color[2][]{%
    \errmessage{(Inkscape) Color is used for the text in Inkscape, but the package 'color.sty' is not loaded}%
    \renewcommand\color[2][]{}%
  }%
  \providecommand\transparent[1]{%
    \errmessage{(Inkscape) Transparency is used (non-zero) for the text in Inkscape, but the package 'transparent.sty' is not loaded}%
    \renewcommand\transparent[1]{}%
  }%
  \providecommand\rotatebox[2]{#2}%
  \ifx\svgwidth\undefined%
    \setlength{\unitlength}{198.94085376bp}%
    \ifx\svgscale\undefined%
      \relax%
    \else%
      \setlength{\unitlength}{\unitlength * \real{\svgscale}}%
    \fi%
  \else%
    \setlength{\unitlength}{\svgwidth}%
  \fi%
  \global\let\svgwidth\undefined%
  \global\let\svgscale\undefined%
  \makeatother%
  \begin{picture}(1,0.44792908)%
    \put(0,0){\includegraphics[width=\unitlength,page=1]{slimquad.pdf}}%
    \put(0.02696196,0.41126617){\color[rgb]{0,0,0}\makebox(0,0)[lb]{\smash{$a$}}}%
    \put(0.79845829,0.40588056){\color[rgb]{0,0,0}\makebox(0,0)[lb]{\smash{$b$}}}%
    \put(0.75537303,0.02349836){\color[rgb]{0,0,0}\makebox(0,0)[lb]{\smash{}}}%
    \put(0.75671946,0.02080549){\color[rgb]{0,0,0}\makebox(0,0)[lb]{\smash{$\pi_Z(b)$}}}%
    \put(-0.00400559,0.01138058){\color[rgb]{0,0,0}\makebox(0,0)[lb]{\smash{$\pi_Z(a)$}}}%
  \end{picture}%
\endgroup%

%% file: pushacrosstriangle.pdf_tex
%% Creator: Inkscape inkscape 0.92.2, www.inkscape.org
%% Accompanies image file 'pushacrosstriangle.pdf' (pdf, eps, ps)
\begingroup%
  \makeatletter%
  \providecommand\color[2][]{%
    \errmessage{(Inkscape) Color is used for the text in Inkscape, but the package 'color.sty' is not loaded}%
    \renewcommand\color[2][]{}%
  }%
  \providecommand\transparent[1]{%
    \errmessage{(Inkscape) Transparency is used (non-zero) for the text in Inkscape, but the package 'transparent.sty' is not loaded}%
    \renewcommand\transparent[1]{}%
  }%
  \providecommand\rotatebox[2]{#2}%
  \ifx\svgwidth\undefined%
    \setlength{\unitlength}{254.30689905bp}%
    \ifx\svgscale\undefined%
      \relax%
    \else%
      \setlength{\unitlength}{\unitlength * \real{\svgscale}}%
    \fi%
  \else%
    \setlength{\unitlength}{\svgwidth}%
  \fi%
  \global\let\svgwidth\undefined%
  \global\let\svgscale\undefined%
  \makeatother%
  \begin{picture}(1,0.41839561)%
    \put(0,0){\includegraphics[width=\unitlength,page=1]{pushacrosstriangle.pdf}}%
    \put(-0.00313352,0.13060716){\color[rgb]{0,0,0}\makebox(0,0)[lb]{\smash{$a$}}}%
    \put(0.86055863,0.38971472){\color[rgb]{0,0,0}\makebox(0,0)[lb]{\smash{$b$}}}%
    \put(0.93428831,0.01053287){\color[rgb]{0,0,0}\makebox(0,0)[lb]{\smash{$c$}}}%
    \put(0,0){\includegraphics[width=\unitlength,page=2]{pushacrosstriangle.pdf}}%
  \end{picture}%
\endgroup%

%% file: NJU1.pdf_tex
%% Creator: Inkscape inkscape 0.92.2, www.inkscape.org
%% Accompanies image file 'NJU1.pdf' (pdf, eps, ps)
\begingroup%
  \makeatletter%
  \providecommand\color[2][]{%
    \errmessage{(Inkscape) Color is used for the text in Inkscape, but the package 'color.sty' is not loaded}%
    \renewcommand\color[2][]{}%
  }%
  \providecommand\transparent[1]{%
    \errmessage{(Inkscape) Transparency is used (non-zero) for the text in Inkscape, but the package 'transparent.sty' is not loaded}%
    \renewcommand\transparent[1]{}%
  }%
  \providecommand\rotatebox[2]{#2}%
  \ifx\svgwidth\undefined%
    \setlength{\unitlength}{241.68816044bp}%
    \ifx\svgscale\undefined%
      \relax%
    \else%
      \setlength{\unitlength}{\unitlength * \real{\svgscale}}%
    \fi%
  \else%
    \setlength{\unitlength}{\svgwidth}%
  \fi%
  \global\let\svgwidth\undefined%
  \global\let\svgscale\undefined%
  \makeatother%
  \begin{picture}(1,0.31804533)%
    \put(0,0){\includegraphics[width=\unitlength,page=1]{NJU1.pdf}}%
    \put(-0.00329712,0.12355783){\color[rgb]{0,0,0}\makebox(0,0)[lb]{\smash{$x$}}}%
    \put(0.88332354,0.13574889){\color[rgb]{0,0,0}\makebox(0,0)[lb]{\smash{$y$}}}%
    \put(0,0){\includegraphics[width=\unitlength,page=2]{NJU1.pdf}}%
    \put(-0.00218886,0.25765922){\color[rgb]{0,0,0}\makebox(0,0)[lb]{\smash{$x_A$}}}%
    \put(0.88221522,0.25655095){\color[rgb]{0,0,0}\makebox(0,0)[lb]{\smash{$y_A$}}}%
    \put(0,0){\includegraphics[width=\unitlength,page=3]{NJU1.pdf}}%
    \put(0.38349112,0.27982474){\color[rgb]{0,0,0}\makebox(0,0)[lb]{\smash{$H$}}}%
    \put(0.31145324,0.07701025){\color[rgb]{0,0,0}\makebox(0,0)[lb]{\smash{$z$}}}%
    \put(0,0){\includegraphics[width=\unitlength,page=4]{NJU1.pdf}}%
  \end{picture}%
\endgroup%